\documentclass{amsart}
\usepackage{amssymb}
\usepackage{braket}
\usepackage{mathrsfs}
\usepackage{ifthen}
\usepackage{graphicx}
\usepackage{tikz}
\usetikzlibrary{patterns}
\usepackage{comment} 
\usepackage{mleftright}
\usepackage{hyperref}
\usepackage{cleveref}
 \usepackage[all]{xy}
 \usepackage{amscd}
 \usepackage{amsrefs}
 \usepackage{color}
 \usepackage{comment}
\usepackage{todonotes}
\usepackage{url}
\usepackage{ulem}
\usepackage{slashed}

\usepackage{scalerel,stackengine}
\stackMath
\newcommand\reallywidecheck[1]{%
\savestack{\tmpbox}{\stretchto{%
  \scaleto{%
    \scalerel*[\widthof{\ensuremath{#1}}]{\kern-.6pt\bigwedge\kern-.6pt}%
    {\rule[-\textheight/2]{1ex}{\textheight}}%WIDTH-LIMITED BIG WEDGE
  }{\textheight}% 
}{0.5ex}}%
\stackon[1pt]{#1}{\scalebox{-1}{\tmpbox}}%
}
\theoremstyle{plain}
\newtheorem{theo}{Theorem}[section]
\crefname{theo}{Theorem}{Theorems}
\Crefname{theo}{Theorem}{Theorems}
\newtheorem{prop}[theo]{Proposition}
\crefname{prop}{Proposition}{Propositions}
\Crefname{prop}{Proposition}{Propositions}
\newtheorem{lem}[theo]{Lemma}
\crefname{lem}{Lemma}{Lemmas}
\Crefname{lem}{Lemma}{Lemmas}
\newtheorem{conj}[theo]{Conjecture}
\crefname{conj}{Conjecture}{Conjectures}
\Crefname{conj}{Conjecture}{Conjectures}
\newtheorem{cor}[theo]{Corollary}
\crefname{cor}{Corollary}{Corollaries}
\Crefname{cor}{Corollary}{Corollaries}

\crefname{claim}{Claim}{Claims}
\Crefname{claim}{Claim}{Claims}

\crefname{property}{Property}{Properties}
\Crefname{property}{Property}{Properties}

\crefname{problem}{Problem}{Problems}
\Crefname{problem}{Problem}{Problems}

\theoremstyle{definition}
\newtheorem{defi}[theo]{Definition}
\crefname{defi}{Definition}{Definitions}
\Crefname{defi}{Definition}{Definitions}

\crefname{notation}{Notation}{Notations}
\Crefname{notation}{Notation}{Notations}

\crefname{convention}{Convention}{Conventions}
\Crefname{convention}{Convention}{Conventions}

\crefname{cond}{Condition}{Conditions}
\Crefname{cond}{Condition}{Conditions}

\crefname{assum}{Assumption}{Assumptions}
\Crefname{assum}{Assumption}{Assumptions}

\theoremstyle{remark}
\newtheorem{rem}[theo]{Remark}
\crefname{rem}{Remark}{Remarks}
\Crefname{rem}{Remark}{Remarks}
\newtheorem{ex}[theo]{Example}
\crefname{ex}{Example}{Examples}
\Crefname{ex}{Example}{Examples}
\theoremstyle{question}
\newtheorem{ques}[theo]{Question}
\crefname{ques}{Question}{Questions}
\Crefname{ques}{Question}{Questions}

\crefname{section}{Section}{Sections}
\Crefname{section}{Section}{Sections}
\crefname{subsection}{Subsection}{Subsections}
\Crefname{subsection}{Subsection}{Subsections}
\crefname{figure}{Figure}{Figures}
\Crefname{figure}{Figure}{Figures}

\newcommand{\Z}{\mathbb{Z}}
\newcommand{\N}{\mathbb{N}}
\newcommand{\R}{\mathbb{R}}
\newcommand{\C}{\mathbb{C}}

\newcommand{\Q}{\mathbb{Q}}
\newcommand{\CP}{\mathbb{CP}}
\newcommand{\RP}{\mathbb{RP}}

\newcommand{\pt}{\mathrm{pt}}

\newcommand{\calV}{{\mathcal V}}

\newcommand{\calW}{{\mathcal W}}

\newcommand{\fraks}{\mathfrak{s}}
\newcommand{\frakt}{\mathfrak{t}}

\newcommand{\id}{\mathrm{id}}
\newcommand{\ind}{\mathop{\mathrm{ind}}\nolimits}

\newcommand{\wt}{\widetilde}

\def\Pin{\mathrm{Pin}}

\def\C{\mathbb{C}}

\def\wt{\widetilde}

\newcommand{\Diff}{\mathrm{Diff}}
\newcommand{\Homeo}{\mathrm{Homeo}}

\newcommand{\del}{\partial}

\title[Involutions, links, and Floer cohomologies]{Involutions, links, and Floer cohomologies}

\author{Hokuto Konno}
\address{Graduate School of Mathematical Sciences, the University of Tokyo, 3-8-1 Komaba, Meguro, Tokyo 153-8914, Japan}
\email{konno@ms.u-tokyo.ac.jp}

\author{Jin Miyazawa}
\address{Graduate School of Mathematical Sciences, the University of Tokyo, 3-8-1 Komaba, Meguro, Tokyo 153-8914, Japan}
\email{miyazawa@ms.u-tokyo.ac.jp}
\author{Masaki Taniguchi}
\address{Department of Mathematics, Graduate School of Science, Kyoto University, Kitashirakawa Oiwake-cho, Sakyo-ku, Kyoto 606-8502, Japan}
\email{taniguchi.masaki.7m@kyoto-u.ac.jp}

\begin{document}

\maketitle

\begin{abstract}
We develop a version of Seiberg--Witten Floer cohomology/homotopy type for a spin$^c$ 4-manifold with boundary and with an involution which reverses the spin$^c$ structure, as well as a version of Floer cohomology/homotopy type for oriented links with non-zero determinant. This framework generalizes the previous work of the authors regarding Floer homotopy type for spin 3-manifolds with involutions and for knots. 
Based on this Floer cohomological setting, we prove Fr\o yshov-type inequalities which relate topological quantities of 4-manifolds with certain equivariant homology cobordism invariants. %Also, we prove that our Floer homotopy types for fibered knots are equivariantly spherical. 
The inequalities and homology cobordism invariants have applications to the %fibered-ness of knots, the 
topology of unoriented surfaces, the Nielsen realization problem for non-spin 4-manifolds,
and non-smoothable unoriented surfaces in 4-manifolds. 
%Based on our Floer cohomology, we provide a new homomorphism from the knot concordance group to $\Z$ which is not a slice torus invariant. 
\end{abstract}

\tableofcontents

\section{Introduction}

\label{section Introduction}

\subsection{Background}

We introduce a version of Seiberg--Witten Floer cohomology theory for 3-manifolds with involutions, as well as a version of Seiberg--Witten Floer cohomology theory for oriented links with non-zero determinant, which can be seen as a generalization of a theory developed by the authors in \cite{KMT}.

In \cite{KMT}, for a given oriented 3-manifold $Y$ with a spin structure $\frak{s}$ and an odd involution $\iota$, the authors constructed Seiberg--Witten Floer Floer homotopy theory/$K$-theory and proved a 10/8-type inequality in the setting, based on Manolescu's Seiberg--Witten Floer homotopy type \cite{Ma03}.

In this paper, we extend such a theory for (possibly non-spin) spin$^c$ structures called {\it real spin$^c$ structures} defined as follows.
A real spin$^c$ structure on a manifold is 
a pair of a spin$^c$ structure $\fraks$ and an orientation-preserving smooth involution $\iota$ on the manifold such that 
\[
\iota^* \fraks \cong  \overline{\fraks}, 
\]
where $\overline{\fraks}$ is the conjugate spin$^c$ structure of $\frak{s}$.
We call a manifold equipped with a real spin$^c$ structure a {\it real spin$^c$ manifold}.
In this paper, we mainly consider involutions that have codimension-2 non-empty fixed sets.

For a real spin$^c$ 3- or 4-manifold, we shall define an involutive symmetry $I$ of the Seiberg--Witten equations acting anti-linearly on the domain and the codomain of the map corresponding to the equations. When a spin$^c$ structure is spin, the involution $I$ is the same as the involution $I= j\circ \tilde{\iota}$ used in previous paper \cite{KMT}, where $\tilde{\iota}$ is a lift of $\iota$ to a spin structure and $j$ denotes the third basis vector of the quaternion. We shall use information of Seiberg--Witten theory corresponding to the fixed point set of $I$.  

There are several preceding studies of {Seiberg--Witten} theory for real spin$^c$ 3- or 4-manifolds, which may be called {\it real Seiberg--Witten theory}. 
In \cite{TW09}, Tian and Wang introduced real Seiberg--Witten invariants for hermitian
almost complex 4-manifolds by counting the moduli spaces corresponding to the $I$-fixed point parts. Recently, Li \cite{JL22} introduced real monopole Floer homology for real spin$^c$ 3-manifolds, which can be seen as the associated Floer homology to consider the relative version of real Seiberg--Witten invariants for real 4-manifolds in \cite{TW09},  by developing Floer theory for the $I$-fixed point parts. 

It is natural to conjecture that our framework gives a Floer homotopy refinement of Li's real monopole Floer homology for real spin$^c$ 3-manifolds under the assumption that $b_1=0$, as Manolescu's Seiberg--Witten Floer homotopy type gives a spacial refinement of monopole Floer homology \cite{LM18}. 

While we shall mainly focus on involutions with codimension-2 non-empty fixed point sets, the corresponding theory for free involution has been known as $\Pin^-(2)$ Seiberg--Witten theory developed by Nakamura \cite{Na13, Na15}. 
Also, when we consider a real spin structure, the corresponding Floer homotopy type or 4-dimensional invariant can be seen as the fixed point parts of 
Montague's \cite{Ian22} $\Z_4 \times_{\Z_2} \Pin(2)$-equivariant homotopy spectra $SWF(\Sigma(L), \mathfrak{s})$ and $\Z_4 \times_{\Z_2} \Pin(2)$-equivariant Bauer--Furuta invariant with respect to the element $[(i, j)]$, where $i$ is a generator of $\Z_4$.

Our main results are Fr\o yshov-type inequalities for real spin$^c$ 4-manifolds with boundary, and for surfaces bounded by links. 
% These inequalities have applications to non-smoothable surfaces in 4-manifolds, the Nielsen realization problem, and the topology of non-orientable surfaces bounded by torus knots. 
 We first exhibit results for closed 4-manifolds below. 
\subsection{A constraint on involutions on closed 4-manifolds}

Let $(W, \fraks,\iota)$ be a real spin$^c$ 4-manifold, i.e. $W$ is an oriented smooth 4-manifold, $\fraks$ is a spin$^c$ structure on $W$, and $\iota : W \to W$ is an orientation-preserving smooth involution such that $\iota^{\ast}\fraks \cong \bar{\fraks}$.
We shall define the notion that $\iota$ is of {\it odd type}, and see that $\iota$ is of odd type if the fixed-point set $W^\iota$ is non-empty and of codimension-2.
Let $\sigma(W)$ denote the signature of $W$ and $b^+(W) $ denote the maximal dimension of positive-definite subspaces of $H^2(W; \R)$.
We denote by $b^+_\iota(W)$ the maximal dimension of the $\iota$-invariant part of positive-definite subspaces of $H^2(W; \R)$.
One of our main theorems is stated as follows: 

\begin{theo}
\label{theo: main}
Let $(W, \fraks, \iota)$ be a closed oriented smooth real spin$^c$ 4-manifold with $b_1(W)=0$.
Suppose that $\iota$ is of odd type.
If $b^{+}_{\iota}(W) = b^{+}(W)$, then we have
\begin{align}
\label{eq: closed ineq}
c_{1}(\fraks)^{2} - \sigma(W) \leq 0.
\end{align}
\end{theo}

For negative-definite 4-manifolds, the inequality \eqref{eq: closed ineq} is well-known to be held for every spin$^c$ structure. The theorem states that we have the inequality \eqref{eq: closed ineq} also for non-negative-definite 4-manifolds, as far as $b^{+}_{\iota}(W) = b^{+}(W)$ and $\fraks$ is real.

\begin{rem}
One can apply \cref{theo: main} to 4-manifolds with involution obtained from double branched covering spaces along surfaces in 4-manifolds. 
By describing \eqref{eq: closed ineq} in terms of the base 4-manifolds, we obtain the following from \cref{theo: main}: for a closed 4-manifold $X$ and a smoothly embedded (possibly non-orientable) surface $S$ satisfying 
\[
[S] \equiv 0 \operatorname{mod} 2 \text{ and }  b^+(X) + \frac{1}{2} b_1(S) -\frac{1}{4} S \circ S=0, 
\]
one has 
\begin{align}\label{real ineq for surface}
c_1(\mathfrak{s})^2 -2 \sigma (X) + \frac{1}{2} S \circ S \leq 0
\end{align}
for any spin$^c$ structure on the double branched covering space $\Sigma(S)$ which satisfies $\iota^* \fraks \cong \overline{\fraks}$, where $\iota$ is the covering involution on $\Sigma(S)$ and $S\circ S$ denotes the evaluation of the normal Euler class of $S$.
 \end{rem}

\subsection{New link concordance invariants}
\cref{theo: main} is generalized to the case of 4-manifolds with boundary and it will be translated as constraints of surfaces embedded into 4-manifolds via the double branched cover construction. In order to write Fr\o yshov type inequality for 4-manifolds with boundary, we need to introduce Fr\o yshov type invariants. 
As such a Fr\o yshov invariant, we shall define a {\it link invariant} 
\[
\delta_R (L) \in  \frac{1}{16}\Z
\]
for an oriented link with non-zero determinant, where $\Sigma(L)$ is the double branched covering space along $L$ and $\frakt_L$ is the spin structure corresponding to the orientation of $L$. Here $R$ stands for {\it real Seiberg--Witten Fr\o yshov invariant}.
In order to summarize basic properties of $\delta_R (L)$, it is convenient to use notion of {\it $\chi$-concordance} introduced in  \cite{DO12}. 
 Let us first review the definition of $\chi$-concordance. A {\it marked link} is a link in $S^3$ equipped with a marked component. 
For given oriented marked links $L_0$ and $L_1$, we call $L_0$ and $L_1$ are {\it $\chi$-concordant} if $-L^*_0\#L_1$ bounds a smoothly properly embedded surface $F$ in $D^4$
such that
\begin{itemize}
    \item[(i)]  $F$ is a disjoint union of one disk together with annuli and M{\"o}bius bands;
    \item[(ii)] the boundary of the disk component of $F$ is the marked component of $-L^*_0 \# L_1$;
    \item[(iii)] we require orientable components of $F$ to be oriented compatible with the orientation of $-L^*_0 \#L_1$ , 
\end{itemize}
where $L^*$ means the mirror image of $L$ and the connected sum $-L^*_0\#L_1$ is taken along marked components. In \cite{DO12}, it is proven that the set $\tilde{\mathcal{L}}$ of all $\chi$-concordant classes of  oriented marked links forms an abelian group with respect to the connected sum along marked components. The group $\tilde{\mathcal{L}}$ is called the {\it link concordance group}. 
In this paper, we focus on the subgroup $\tilde{\mathcal{F}}$ of $\tilde{\mathcal{L}}$ generated by oriented marked links whose determinants are {\it non-zero}.
For an embedded surface $S$ in a given 4-manifold $X$, we denote by $\Sigma(S)$ the double branched cover of $X$ along $\Sigma$ if it exists.

\begin{theo}
\label{theo: main knot}
The invariant $\delta_R (L)$ associated to an oriented link with non-zero determinant satisfies the following properties: 
\begin{itemize}
    \item[(i)]  The quantity $\delta_R (L) $ is a $\chi$-concordance invariant. Thus, we have a well-defined map  $\delta_R: \tilde{\mathcal{F}} \to \frac{1}{16}\Z$. 
\item [(ii)] The map $\delta_R: \tilde{\mathcal{F}} \to \frac{1}{16}\Z$ is a homomorphism. 
    
    \item[(iii)] Let $L$ and $L'$ be oriented links in $S^3$ with non-zero determinants, $X$ be an oriented smooth compact connected cobordism from $S^3$ to $S^3$, and $S$ be a compact connected properly and smoothly (possibly non-orientable) embedded cobordism in $X$ from $L$ to $L'$ such that the homology class $[S] \equiv 0 \operatorname{mod} 2$. Suppose that $b_1(X)=0$, 
    \begin{align} \label{assumption delta}
    b^+(X) + \frac{1}{2} b_1(S) -\frac{1}{4} S \circ S  -  \frac{1}{2} \sigma (L) + \frac{1}{2} \sigma (L')=0,
    \end{align}
    and there is a spin$^c$ structure $\fraks$ on $\Sigma (S)$ such that $\iota^* \fraks \cong \overline{\fraks}$ and whose restrictions are compatible with orientations of $L$ and $L'$.
    
    Then, we have
    \begin{align} \label{ineq1}
     \delta_R (L)+
\frac{1}{16}\left( \langle c_{1}(\fraks)^{2}, [\Sigma (S)] \rangle - 2 \sigma (X) - \frac{1}{2} S \circ S - \sigma (L)+ \sigma (L')\right) \leq \delta_R(L')
    \end{align}
    where $\sigma(L)$ denotes the signature of $L$ (with the sign convention $\sigma(T(2,3))=-2$),  and $S \circ S$ means the self-intersection number (or normal Euler number) of $S$ in $X$.
\end{itemize}
\end{theo}
As an advantage of the inequality \eqref{ineq1}, which we may call {\it real Fr\o yshov inequality}, we can directly treat non-orientable surfaces. 
\begin{rem}
  The inequality \eqref{ineq1} implies \eqref{real ineq for surface} if we put $L$ and $L'$ are the unknot. Also note that the quantity $b^+(X) + \frac{1}{2} b_1(S) -\frac{1}{4} S \circ S  -  \frac{1}{2} \sigma (L) + \frac{1}{2} \sigma (L')$ is non-negative since it can be written as $b^+(\Sigma(S)) - b^+_{\iota} (\Sigma(S))$.
\end{rem}

\begin{rem}
When we restrict $\delta_R$ as a function on the knot concordance group, we obtain a homomorphism 
\[
\delta_R : \mathcal{C} \to \frac{1}{16} \Z. 
\]
For two bridge knots, we have $\delta_R = -\frac{1}{16}\sigma$. It is natural to ask $\delta_R$ is a slice torus invariant (\cite{Li04,Lew14}) up to constant multiplication. 
However, \eqref{ineq1} does not imply slice genus bounds and, for torus knots, it is {\it not} equal to their slice genera up to constant multiplication. 
\end{rem}

As analog of invariants introduced by Stoffregen~\cite{Sto171}   $\bar{\delta}(Y,\frakt),\  \underline{\delta}(Y,\frakt)$ corresponding to $\bar{d}(Y,\frakt),\ \underline{d}(Y,\frakt)$ in involutive Heegaard Floer homology \cite{HM17}, we also define invariants
\[
\bar{\delta}_R(L)  \text{ and }\underline{\delta}_R(L)\in  \frac{1}{16}\Z 
\]
for any link $L$ with non-zero determinant. 

When $\frak{s}$ is spin and \eqref{assumption delta} is satisfied, we also have inequalities similar to \eqref{ineq1}. 
See \cref{theo: main ineq d bar} for the details.
Moreover, we also have the following Theorem B type inequality:

\begin{theo}\label{Theorem B for links}
The invariants $\underline{\delta}_R (L)$ and $\bar{\delta}_R (L)$ associated to an oriented link with non-zero determinant satisfies the following property: 
Let $L, L', X, S$ be as in (iii) of \cref{theo: main knot}, except for that we suppose that $\fraks$ is a spin structure and that
\begin{align} \label{assumption delta bar}
b^+(X) + \frac{1}{2} b_1(S) -\frac{1}{4} S \circ S  -  \frac{1}{2} \sigma (L) + \frac{1}{2} \sigma (L')=1,
\end{align}
instead of \eqref{assumption delta}.   
Then, we have
    \begin{align} \label{ineq2}
     \underline{\delta}_R(L)+
\frac{1}{16}\left( - 2 \sigma (X) + \frac{1}{2} S \circ S + \sigma (L)- \sigma (L')\right) \leq \bar{\delta}_R(L'). \end{align}
\end{theo}

\begin{rem}
If we take the unknot as both $L$ and $L'$ and the trivial null-homologous $2$-disk embedded in a spin 4-manifold $X$ as $S$, then the inequality \eqref{ineq2} implies $
- \frac{1}{8 } \sigma (X) \leq 0 
$
if $b^+(X)=1$. This recovers Theorem B proven by Donaldson \cite{Do02}. 
\end{rem}
By specifying \cref{theo: main knot} to $X=S^3 \times [0,1]$, $L'=U$ and $L$ is a knot, we have the following constraint on topology of non-orientable surfaces in $D^4$: 
\begin{theo}\label{unoriented bound}
Let $S$ be a (possibly non-orientable) smoothly and properly embedded connected surface in $D^4$ bounded by a given knot $K$. 
Suppose that
\begin{itemize}
    \item[(i)] the determinant of $K$ is $1$;
    \item[(ii)] the Manolescu--Owen invariant \cite{MO07} $\delta(K) = \delta(\Sigma (K))$ is zero;
    \item[(iii)] $\delta_R(K)<0$. 
\end{itemize}
Then, we have 
    \[
    -\sigma(K)+ \frac{1}{2}S \circ S + 1 \leq b_1(S). 
    \]
\end{theo}
We shall use \cref{unoriented bound} to obstruct certain types of non-orientable surfaces bounded by torus knots in the next subsection. 

\begin{rem}\label{crossing change}
As a corollary of \cref{theo: main knot}, one can provide crossing change formula for $\delta_R (L)$, $\bar{\delta}_R(L)$ and $\underline{\delta}_R(L)$. Let us summarize the results for the case of knots: 
Let $K$ and $K'$ be knots in $S^3$. Suppose that $K'$ is obtained from $K$ by a positive crossing change, where our convention on the positive crossing change is the same as in \cite{KMT}. 
Then, we have the following: 
\begin{itemize}
    \item[(i)] Suppose $\sigma (K)- \sigma (K')=2$. 
Then the inequalities 
\begin{align*}
\delta_R( K)+  \frac{1}{8}\leq \delta_R(K'), \ \ 
\underline{\delta}_R( K)+  \frac{1}{8}\leq \underline{\delta}_R(K'), \ \
\bar{\delta}_R( K)+  \frac{1}{8}\leq \bar{\delta}_R(K') 
\end{align*}
hold.
\item[(ii)] Suppose $\sigma (K)- \sigma (K')=0$. Then the inequalities
\begin{align*}
\delta_R( K')-  \frac{1}{4}\leq \delta_R(K), \ \  
\underline{\delta}_R( K')-  \frac{1}{4}\leq \underline{\delta}_R(K), \ \ 
\bar{\delta}_R( K')-  \frac{1}{4}\leq \bar{\delta}_R(K)  
\end{align*}
hold.
\end{itemize}
\end{rem}

\begin{comment}

\begin{theo}We have the following two band surgery formula: 

\begin{itemize}
    \item[(i)] (non-coherent band surgery) 
Suppose $K$ and $K'$ are knots in $S^3$ with the same determinant. %{\color{red} Try to relax this condition }
Suppose there is a band surgery $b$ from $K$ to $K'$ such that $b\cdot b=1$. Then the inequality 
\[
\delta(K)\leq \delta(K')
\]
holds. 
\item[(ii)] (coherent band surgery) 
\end{itemize}
\end{theo}

\begin{theo} \label{crossing change}
Let $K$ and $K'$ be knots in $S^3$. Suppose $K'$ is obtained from $K$ by a positive crossing change. 

\begin{itemize}
    \item[(i)] Suppose $\sigma (K)- \sigma (K')=2$. 
Then the inequalities 
\begin{align*}
\delta( K)+  \frac{1}{8}\leq \delta(K'), \ \ 
\underline{\delta}_R( K)+  \frac{1}{8}\leq \underline{\delta}_R(K'), \ \
\bar{\delta}_R( K)+  \frac{1}{8}\leq \bar{\delta}_R(K') 
\end{align*}
hold, where our convention on the positive crossing change is the same as in \cite{KMT}. 
\item[(ii)] Suppose $\sigma (K)- \sigma (K')=0$. Then the inequalities
\begin{align*}
\delta( K')-  \frac{1}{4}\leq \delta(K), \ \  
\underline{\delta}_R( K')-  \frac{1}{4}\leq \underline{\delta}_R(K), \ \ 
\bar{\delta}_R( K')-  \frac{1}{4}\leq \bar{\delta}_R(K)  
\end{align*}
hold.
\end{itemize}
\end{theo}

\end{comment}

We compute $\delta_R (L), \bar{\delta}_R(L), \underline{\delta}_R(L)$ for some class of links:

\begin{prop}\label{computation}We have the following computations: 
\begin{itemize}
    \item[(i)] For a two bridge link $K(p,q)$, we have 
    \[
    \delta_R(K(p,q)) = \underline{\delta}_R(K(p,q))  = \bar{\delta}_R(K(p,q)) =-\frac{\sigma(K(p,q))}{16}.
\]
    \item[(ii)] Let $p, q$ be coprime odd integers and $T(p, q)$ be the torus knot of
type $(p, q)$. Then we have  \[
    \delta_R (T(p,q)) = \underline{\delta}_R (T(p,q))= \bar{\delta}_R (T(p,q)) = -\frac{{\bar{\mu}} (\Sigma (2,p,q) )}{2},
    \]
where $\bar{\mu}$ denote the Neumann--Siebenmann invariant \cite{N80}.
\end{itemize}
\end{prop}

Also, the following computations follow from Montague's computations of equivariant Seiberg--Witten Floer stable homotopy types \cite{Ian22}. 
 
\begin{ex}
   For Montesinos knots of type $M(2,3, 6n \pm 1)$, we have the following computations (see \cref{ex Brieskorn sphere with Montesinos}): 
   \begin{center}
\begin{tabular}{| c | |  c | c | c | c |}
\hline
Montesinos knot & {$\bar{\delta}_R$} & {$\delta_R$} & {$\underline{\delta}_R$} \\
\hline
\hline
$M(2,3,12k-1)$  & $1/2$ & $1/2$ & $0$  \\
\hline
$M(2,3,12k-5)$  & $0$ & $0$ & $-1/2$  \\
\hline
$M(2,3,12k+1)$  & $0$ & $0$ & $0$  \\
\hline
$M(2,3,12k+5)$  & $1/2$ & $1/2$ & $1/2$ \\
\hline
\end{tabular}
\end{center} 
\end{ex}

Now, we focus on applications of \cref{theo: main} and \cref{theo: main knot}.

\subsection{Applications to non-orientable surfaces in $D^4$ bounded by torus knots}

Recently, using several types of Floer theories, the {\it non-orientable 4-genus} has been studied, for example see \cite{B14, Lo19, Al20,BKST21,Jo22}. 
For a given knot $K$ in $S^3$, we focus on topological types of smoothly and properly embedded possibly non-orientable surfaces in $D^4$ bounded by $K$. For such a surface $S$, the {\it non-orientable genus } $h(S)$ of $S$ is defined by $h(S):= b_1(S)$. The non-orientable 4-genus is the minimum number of $h(S)$ for all such $S$. Moreover, there is also another topological invariant of the embedding $S$, called the {\it normal Euler number} $S \circ S$. The number $S \circ S$ is defined as the evaluation of the twisted Euler class of a normal bundle of $S$ by the fundamental class of $S$ (in the orientation local system of $S$). The following question is natural to ask regarding topology of non-orientable surfaces:
\begin{ques}[\cite{Al20}]
\label{ques} For a given knot $K$ in $S^3$, what is the set of realizable pairs
\[
(S \circ S, h(S)) \in \Z \times \Z_{>0},
\]
where $S$ are properly and smoothly embedded (possibly non-orientable) surfaces in $D^4$ bounded by $K$.
\end{ques}
Even for torus knots, \cref{ques} has not been solved completely.
We focus on torus knots of type $T(3,q)$.
By using the Heegaard Floer $d$-invariant and the upsilon invariant $\Upsilon$ \cite{OSS17}, Allen~\cite{Al20} detected the set all such possible pairs $(S \circ S, h(S))$ for positive torus knots $T(3, 6n+4)$ and $T(3,6n+5)$.
By a similar argument, she \cite{Al20} detected  the realizable pairs $(S \circ S, h(S))$ for $T(3, 6n+1)$ and $T(3,6n+2)$ except for the following cases:
\begin{itemize}
    \item[(i)] for $T(3, 6n+1)$, 
    $
    (e, h) =(8/3 (1-n) +2+2m  ,1+m);
    $
    \item[(ii)]for $T(3, 6n+2)$, 
    $
    (e, h) =(8/3 (2-n) +2+2m  ,3+m)$, 
\end{itemize}
where $m$ is a non-negative integer.
Related the remaining parts, Allen conjectured the following.
\begin{conj}[\cite{Al20}]
Both of (i) and (ii) are {\it not} realizable. 
\end{conj}
\cref{unoriented bound} enables us to prove the half of the conjecture. 
\begin{theo}\label{torus non orientable}
The case (i) is not realizable. 
\end{theo}
Combined with Allen's result, we can detect all realizable pairs. 
\begin{cor}
The set of realizable pairs $(S \circ S, h(S))$ for $T(3, 6n+1)$ is given by
\[
(-2  - 16n \pm 2m, 1 + m + 2l) , \ m,l \geq 0. 
\]
\end{cor}

\subsection{Nielsen realization for non-spin 4-manifolds}

Next application is about the Nielsen realization problem.
Given a smooth manifold $W$, a subgroup $G$ of $\pi_0(\Diff(W))$ is said to be {\it realizable in $\Diff(W)$} when there is a section $s : G \to \Diff(W)$ of the natural map $\Diff(W) \to \pi_0(\Diff(W))$ over $G$.
Recall that, for a $(-1)$-sphere $S$ in an oriented 4-manifold $W$, we obtain an orientation-preserving diffeomorphism $\rho_{S} : W \to W$ called the {\it reflection}, which is locally modeled on the complex conjugation on $-\CP^2$.
It is easy to see that
the diffeomorphism $\rho_{S}$ generates an order 2 subgroup of $\pi_{0}(\Diff(W))$.
(See \cref{lem: reflection order 2}.)

\begin{theo}
\label{theo: Nielsen}
Let $W$ be a closed oriented smooth 4-manifold with $\sigma(W) \neq -2$.
Let $\fraks$ be a spin$^{c}$ structure on $W$ and let $S$ be a smoothly embedded $(-1)$-sphere in $W$.
If $c_{1}(\fraks)^{2} - \sigma(W)>0$ and $\rho_{S}^{\ast}\fraks \cong \bar{\fraks}$, then $\rho_{S}$ is not homotopic to any smooth involution.
\end{theo}

\begin{ex}
\label{ex: blow up}
Let $W'$ be a closed oriented spin smooth 4-manifold with $\sigma(W')<0$.
Set $W = W' \# (-\CP^{2})$.
Let $S$ be the canonical $(-1)$-sphere of the $(-\CP^{2})$-component.
Then $\rho_{S} : W \to W$ is not homotopic to any smooth involution.
Indeed, let $\fraks$ be a spin$^{c}$ structure on $W$ such that $c_{1}(\fraks)|_{W'}=0$ and $c_{1}(\fraks)|_{-\CP^{2}}$ is a generator of $H^{2}(-\CP^{2})$. 
Then the reflection $\rho_{S} : W \to W$ satisfies the assumption of \cref{theo: Nielsen} for $\fraks$.
\end{ex}

\cref{ex: blow up} immediately implies:

\begin{cor}
\label{cor: Nielsen blowup}
Let $W'$ be a closed oriented spin smooth 4-manifold with $\sigma(W')<0$.
Set $W = W' \# (-\CP^{2})$ and let $S$ be the exceptional sphere in the $(-\CP^2)$-component.
Then the order 2 subgroup $G$ of $\pi_0(\Diff(W))$ generated by the mapping class $[\rho_S]$ of the reflection about $S$ is not realizable in $\Diff(W)$.
\end{cor}

To the best of the author's knowledge, \cref{cor: Nielsen blowup} gives the first example of a non-spin 4-manifold that admits non-realizable finite subgroup of the mapping class group.
The first example of a 4-manifold that is shown to admit a non-realizable finite subgroup is some nilmanifold \cite{RS77} due to Raymond and Scott.
(Note that every nilmanifold is parallelizable and hence spin.)
Recently, many simply-connected spin 4-manifolds are shown to admit non-realizable finite subgroups: for $K3$ by Baraglia and the first author \cite{BK21} and by Farb and Looijenga \cite{FL21}, and later for more general spin 4-manifolds by the first author \cite{K22}.

\begin{rem}
After the first version of this paper appeared on arXiv, the authors were informed that, in their upcoming work \cite{AB23}, Arabadji and Baykur give other examples of non-liftable subgroups of mapping class groups of non-spin 4-manifolds, including irreducible 4-manifolds and definite 4-manifolds.   
\end{rem}

\begin{rem}
Generalizing \cref{cor: Nielsen blowup},
one may also obtain a non-realizable order 2 element of $\pi_0(\Diff(W))$ for $W=W'\#n(-\CP^2)$ for a spin $W$ and for a general $n \geq 1$ by considering connected sums of $\rho_S$ along fixed points.
\end{rem}

\begin{rem}
The assumption that $c_1(\fraks)^2 - \sigma(W)>0$ in \cref{theo: Nielsen} cannot be dropped in general.
To see this, let us consider $-\CP^2$. 
The model reflection $\rho_S : -\CP^2 \to -\CP^2$ about the exceptional sphere $S$ is a smooth involution, while any spin$^c$ structure $\fraks$ on $-\CP^2$ is reversed by $\rho_S$.
This is consistent in \cref{theo: Nielsen}
as $c_1(\fraks)^2 - \sigma(-\CP^2)\leq0$. 

More generally, let $f$ be an orientation-preserving diffeomorphism of $W=\CP^{2} \# n(-\CP^{2})$ with $n \leq 8$ for which $f_\ast : H_2(W) \to H_2(W)$ is an involution.
Then it follows from a result by Lee~\cite[Corollary~1.5, Remark~1.7]{Lee22} that $f$ is topologically isotopic (hence homotopic) to some smooth involution.
This is consistent in \cref{cor: Nielsen blowup}, as $\CP^2\#n(-\CP^2)$ is not diffeomorphic to a manifold of the form $W'\#(-\CP^2)$ with spin $W'$ with  $\sigma(W')<0$. 
\end{rem}

We also give a comparison result on the smooth and topologival Nielsen realization problems.
As a topological version of the above realization problem, given a subgroup $G$ of $\pi_0(\Homeo(W))$, we say that $G$ is realized in $\Homeo(W)$ if there is a section of the natural map $\Homeo(W) \to \pi_0(\Homeo(W))$ over $G$.
One may ask whether there is  a discrepancy between the realization problems in the smooth and topological categories.
Such a comparison result was obtained first by Baraglia and the first author~\cite[Theorem~1.2]{BK21}, and it was generalized in \cite[Theorem~1.3]{K22} by the first author.
These results treated only spin 4-manifolds.
The following theorem is the first comparison result for a non-spin 4-manifold:

\begin{theo}
\label{theo: smooth vs. top Nielsen}
For $p>0$ and $q, r \geq 0$, set $W=pK3\# qS^2\times S^2 \# r(-\CP^2)$.
Then there exists a subgroup 
$G$ of $\pi_0(\Diff(W))$ of order 2 that satisfies the following properties:
\begin{itemize}
\item The group $G$ is not realized in $\Diff(W)$.
\item Define a subgroup $G' \subset \pi_0(\Homeo(W))$ to be the image of $G$ under the map $\pi_0(\Diff(W)) \to \pi_0(\Homeo(W))$.
Then $G'$ is non-trivial and not realized in $\Homeo(W)$.
\end{itemize}
\end{theo}

By the use of \cref{theo: main}, we provide the two applications: regarding non-smoothable and non-orientable surfaces and the Nielsen realization problem.

\subsection{Non-smoothable and non-orientable surfaces}

While non-smoothable orientable surfaces have been well studied by the use of Theorem A, $10/8$ inequality and adjunction inequality combined with Freedman's theory (for example, see \cite[
Proof of Lemma 9.4.2 and Addendum 9.4.4]{GS99}), to the best of our knowledge, construction of non-smoothable and non-orientable surfaces have not been developed well. By combining Rochlin's theorem, Theorem A, or $10/8$ inequality with surgery technique, several constraints have been proven \cite{Ma69, Lo85, LRS15}. 
 In \cite{LRS15}, Levine--Ruberman--Strle gave genus bounds when the ambient 4-manifold is definite or the embedded surface is characteristic.
%We provide an example of a non-smoothable locally flat embedding of connected sums of the real projective surface into 4-manifolds of the forms $m \CP^2 \# (-n {\CP}^2)$.
Below we give non-orientable and non-smoothable surfaces, based on a constraint on surfaces that is not characteristic and the ambient 4-manifold is not definite:

\begin{theo}
\label{theo: unorientable surfaces}
Let $m, k >0$ and $n \geq 0$ and set
\[
X = (m+n)\CP^{2} \# (-n-8)\CP^{2}.
\]
Then there exists a locally flat embedding of $k\RP^{2}$ into $X$ that is not topologically isotopic to any smoothly embedded surface.
More precisely: 
\begin{itemize}
\item  The homology class of the image of the embedding of $k\RP^{2}$ is zero in $H_2(X;\Z/2)$.
%\begin{align}
%\label{eq: fixed-point set main appl}
%\begin{split}
%&(2)^{\oplus m} \oplus (0)^{\oplus n} \oplus (0)^{n+8} \\
%\in &
%H_2(\CP^2)^{\oplus m} \oplus H_2(\CP^2)^{\oplus n} \oplus H_2(-\CP^2)^{n+8} = H_2(X).
%\end{split}
%\end{align}
In particular, the image of the embedding is not a characteristic surface, i.e. $[k\RP^2] \neq w_2(X)$ in $H^2(X;\Z/2)$.
\item The normal Euler number is given by $4m + 2k$.
\end{itemize}
\end{theo}

Let us remark that \cref{theo: unorientable surfaces} is almost a rephrase of the following application to detect non-smoothable involutions:

\begin{theo}
\label{theo: nonsmoothable action}
Let $m, k >0$ and $n \geq 0$ and set
\[
W = (m+2n)\CP^{2}\#(-m-2n-k-16)\CP^{2}. 
\]
Then there exists an orientation-preserving non-smoothable locally linear involution $\iota : W \to W$ with 
$b^{+}_{\iota}(W)=m$, $b^{-}_{\iota}(W)=n+8$. 
\end{theo}

\begin{rem}
Non-smoothable involutions on spin 4-manifolds have been extensively studied by Nakamura~\cite{Na09}, Kato~\cite{Ka17} and Baraglia~\cite{B19}.
Baraglia~\cite[Proposition 7.3]{B19} gave also a homological constraint on smooth involution on connected sums of copies of $\CP^{2}$ and $-\CP^{2}$.
However, this constraint works for an involution $\iota$ on $W$ with $b^{+}_{\iota}(W)=0$, which is complementary to \cref{theo: nonsmoothable action}.
\end{rem}

\begin{comment}
All known computations of these invariants satisfy $\delta=\bar{\delta}_R= \underline{\delta}_R$. However, the original invariants $\delta, \bar{\delta}_R,\underline{\delta}_R$ for homology 3-spheres are {\it not} equivalent in general \cite{Sto171, HM17}. Thus, we shall post the following question: 
\begin{ques}
Is there an oriented link $L$ which does not satisfy $\delta=\bar{\delta}_R= \underline{\delta}_R$?   
\end{ques}
\end{comment}

\begin{comment}

\subsection{Fibered knots and topologically slice knots}

It is known that there are lots of concordance classes with no fibered representative. However, their techniques use Tristram--Levine's signature which vanishes for topologically slice knots. 
Our theory also provides an obstruction being a fibered knot, which will be non-trivial even for topologically slice knots.

\begin{theo}
There exists a topologically slice knot which is not smoothly concordant to any fibered knot.  
\end{theo}

\end{comment}

\subsection{Structural conjecture } 

The invariants $\underline{\delta}_R$ and $\bar{\delta}_R$ are analogs of the invariants of homology 3-spheres $\underline{\delta}$ and $\bar{\delta}$ in Seiberg--Witten theory \cite{Sto171}, which are conjectured to be equal to $\underline{d}$ and $\overline{d}$ in involutive Heegaard Floer homology \cite{HM17}. One can ask if several analogous phenomena also hold for our invariants $\underline{\delta}_R$ and $\bar{\delta}_R$. 
We conjecture the following which can be seen as an analog of \cite[Theorem 1.3]{St17} and \cite[Theorem 1.2]{HHL21} for our invariants. 
\begin{conj}
For any knot $K$,
we have 
\[
\lim_{n \to \infty} \frac{\underline{\delta}_R (\#_n K)}{n} = \delta(K) \text{ and } \lim_{n \to \infty} \frac{\bar{\delta}_R (\#_n K)}{n} = \delta(K) . 
\]
\end{conj}

Moreover, as a `real' version of Manolescu--Lidman's isomorphism \cite{LM18}, we conjecture the following: 

\begin{conj}
    For any oriented link $L$ with non-zero determinant, we have 
    \begin{align*}
    \widehat{HMR}_*(L, \frak{s}_L; \Z_2) &\cong H^{\Z_2}_*(SWF(L); \Z_2), \\
    \reallywidecheck{HMR}_*(L,\frak{s}_L; \Z_2) &\cong \mathrm{c}H^{\Z_2}_*(SWF(L); \Z_2) \\
  \overline{HMR}_*(L,\frak{s}_L; \Z_2) &\cong \mathrm{t}H^{\Z_2}_*(SWF(L); \Z_2), 
    \end{align*}
    where $H^{\Z_2}_*, \mathrm{c}H^{\Z_2}$ and $\mathrm{t}H^{\Z_2}$ are $\Z_2$-equivariant Borel, coBorel and Tate homologies respectively and $HMR^\circ $ are real monopole Floer homologies introduced in \cite{JL22} for the spin structure $\frak{s}_L$ determined by orientation of $L$. 
\end{conj}

\subsection{Structure of the paper} 

We finish off this introduction with an outline of
the contents of this paper. 
In \Cref{involutionI}, we define a symmetry $I$ of Seiberg--Witten equation for a given real spin$^c$ structure. In \Cref{section Floer homotopy type of rational homology 3-spheres with involution and Fro yshov type invariants}, using the $I$-invariant part of Manolescu's Floer homotopy type, we introduce three invariants $d(Y, \frak{s}, \iota)$, $\overline{d}(Y, \frak{s}, \iota)$, and $\underline{d}(Y, \frak{s}, \iota)$ for 3-manifolds with real spin$^c$ structures $(Y, \frak{s}, \iota)$. Moreover, we prove several fundamental properties of these invariants including Fr\o yshov type inequality, connected sum formula, and duality formula. %In \Cref{Results on fibered knots}, we compute our homotopy theoretic invariant for a fibered knot in $S^3$. 
In \Cref{section Floer homotopy type of knots and Fro yshov type invariants}, by setting $Y$ to be the double branched covering space of an oriented link $L$ in $S^3$ with non-zero determinant, we obtain invariants for $ L$: $d(L)$, $\overline{d}(L)$, and $\underline{d}(L)$. Moreover, we prove these invariants are $\chi$-concordant invariants, and $d(L)$ defines a homomorphism on the subgroup of the link concordance group generated by oriented based links with non-zero determinant. In \Cref{section: application}, we prove all applications in the introduction such as non-smoothable and non-orientable surfaces in 4-manifolds (\cref{theo: unorientable surfaces}), obstruction to Nielsen realization problem (\cref{theo:  Nielsen}) and non-orientable surfaces in $D^4$ bounded by torus knots (\cref{torus non orientable}).

\begin{comment}
As an analogue of $HF_{\text{conn}}$ introduced in \cite[Theorem 1.1]{HHL21}, we can also define an oriented (based) link concordace invariant $HM_{\text{conn}} (L)$ if the determinant of $L$ is non-zero, which is a $\Z_2[U]$-module valued invariant and naturally a summand of our Floer cohomology 
$SWF(L)$. 
\begin{theo}
For any oriented non-zero determinant link $L$, the following two conditions are equivalent: 
\begin{itemize}
    \item[(i)] $HM_{\text{conn}} (L)=0$;
    \item[(ii)] $\delta_R (L)= \underline{\delta}_R(L)=\bar{\delta}_R (L)=0$.
\end{itemize}
\end{theo}
\end{comment}

\begin{comment}
\subsection{Applications to non-extendability of involutions}

\begin{theo}
\label{theo intro nonextendable general}
Let $a_{1}, \ldots, a_{n}$ be pairwise coprime natural numbers.
Suppose that $a_{1}$ is an even number.
Set $Y=\Sigma(a_{1}, \ldots, a_{n})$, and define an involution $\iota : Y \to Y$ by
\[
\iota(z_1,z_2, \ldots,z_n)=(-z_1,z_2,\ldots,z_n).
\]
Let $W$ be a compact connected simply connected smooth oriented 4-manifold bounded by $Y$ with $b_1(W)=0$ and $b^+(W)-b^+_{\iota}(W)=0$. 
For a given characteristic c in $H^2(W; \Z )$, the involution $\iota$ cannot extend to $W$ as a smooth involution so that
\begin{align*}
\langle c^2, [W, \partial W] \rangle-\frac{\sigma(W)}{16} 
>  - \frac{\bar{\mu}(Y)}{2}
\end{align*}
and $\iota^* c = -c$.

\end{theo}
\end{comment}

%\marginpar{Hokuto: Somehow Acknowledgement becomes italic and I cannot fix it...}

\subsection{Acknowledgement}
We would like to thank Joshua Sabloff for answering our question about his paper \cite{Jo22}. The authors also thank Kouki Sato, Tye Lidman, Peter Feller, Mike Miller for their helpful discussions. The authors would like to thank Jiakai Li and Ian Montague for enlightening discussions.
The first author was partially supported by JSPS KAKENHI Grant Numbers 19K23412 and 21K13785, and Overseas Research Fellowships.
The second author was supported by JSPS KAKENHI Grant Number 21J22979 and WINGS-FMSP program at the Graduate school of Mathematical Science, the University of Tokyo.
The third author was supported by JSPS KAKENHI Grant Number 20K22319 and RIKEN iTHEMS Program.

\section{Involution $I$}\label{involutionI}
In this section, we explain the definition of the anti-linear involution $I$ which covers the involution $\iota$. 
Let $W$ be an oriented smooth 4-manifold and $\fraks$ be a spin$^{c}$ structure on $W$. 
Let $\iota : W \to W$ be a smooth involution that preserves the orientation of $W$ and satisfies that $\iota^{\ast}\fraks \cong \bar{\fraks}$. We also assume that $W^{\iota} \neq \emptyset$ and $H^1(W, \Z)$ is $0$. 
In this section, we fix an $\iota$-invariant metric on $W$. 

Note that a $\mathrm{spin}^c$ structure $\fraks$ gives rise to a $\Z/2\Z$-graded Clifford module $S=S^+ \oplus S^-$ with a hermitian metric, and $\fraks$ is determined uniquely by $S$ \cite[Section~1.1]{KM07}. 
In this section, we identify $\fraks$ with $S$. 

Firstly, we give an anti-linear involution on $\mathfrak s$ if the involution $\iota$ satisfies a condition which is called odd type. We will prove later that such an involution is unique up to gauge transformations.

\begin{defi}
Let $\mathfrak s$ be a spin$^c$ structure on $W$ and we will denote by $S=S^+ \oplus S^-$ the spinor bundle of $\mathfrak s$ and will denote by $\rho$ its Clifford multiplication. Let $\bar{\mathfrak s}$ be a complex conjugate spin$^c$ structure whose spinor bundle $\bar{S}$ is the conjugate of $S$, while Clifford multiplication is unchanged as real-linear map. We define the pull-back of the spin$^c$ structure $\mathfrak s$, say $\iota^{\ast}\mathfrak s$, as follows: The spinor bundle of $\iota^{\ast}\mathfrak s$ is $\iota^{\ast}S$ as a complex vector bundle. We define its Clifford multiplication by $(\xi, \phi) \mapsto \rho((\iota^{-1})^* (\xi))\phi$ where $\xi \in T^*_x W$ and $\phi \in \iota^{\ast}S_x = S_{\iota(x)}$. To simplify the notation, we write $\iota$ when it should be $\iota^{-1}$ since $\iota$ is an involution. 
\end{defi}

We now give an anti-linear map on $S$ which covers $\iota$ which may not be an involution. 
\begin{defi}
Let $\iota^{\ast} \colon \iota^{\ast}\frak{s} \to \frak{s}$ be a natural map that covers $\iota$. Note that this is a bijection. 
Let us take an isomorphism of the spin$^c$ structures 
$\varphi \colon  \bar{\fraks} \to \iota^{\ast}\fraks$
and complex conjugate $c \colon \frak{s} \to \bar{\frak{s}}$. 
Let us define $I_{\varphi}= \iota^{\ast}\circ \varphi \circ c$. This is an anti-linear map on $S$ which covers the involution $\iota$. 
\end{defi}
    Note that the definition of the isomorphism between the $\mathrm{spin}^c$ structures $\mathfrak{s}$ and $\mathfrak{s}'$ is an isomorphism of complex vector bundle $\varphi$ from the spinor bundle $S$ of $\mathfrak{s}$ to that of $\mathfrak{s}'$ which satisfies $\varphi(\rho(\xi)\phi)=\rho(\xi)\varphi(\phi)$ for $\xi \in T^*_x W$ and $\phi \in S_x$. From the definition of $I_{\varphi}$, we can easily prove the lemma below. 
\begin{lem}\label{IandCliffford}
The anti-linear map $I_{\varphi}$ satisfies
\[
I_\varphi(\rho(\xi)\phi)=\rho(\iota^*(\xi))I_\varphi(\phi)
\]
for all $\xi \in T_x W$ and $\phi \in S_x$. Moreover, $I_{\varphi}$ preserves the hermitian metric on $S$. 
\end{lem}

We will prove that there is a gauge transformation $u$ such that $(u I_{\varphi})^2=\pm 1$. We prove two lemmas for preparation. 
\begin{lem}\label{uphi1}
There exists a gauge transformation $u_{\varphi} \colon W \to U(1)$ such that $I_{\varphi}^2=u_{\varphi}$. 
\end{lem}
\begin{proof}
We see that $I_{\varphi}^2$ is an automorphism of the spin$^c$ structure $\frak{s}$. Thus $I_{\varphi}^2$ is a gauge transformation. 
\end{proof}

\begin{lem}\label{uphi2}
The gauge transformation $u_{\varphi}$ in~\cref{uphi1} satisfies that $u_{\varphi}(\iota^{-1}(x))=\overline{u_{\varphi}(x)}$. 
\end{lem}
\begin{proof}
Note that for $\psi \in \Gamma(\iota^{\ast}\frak{s})$, $\iota^{\ast}(\psi)(x)=\psi(\iota^{-1}(x))$. We see that for all $\phi \in \Gamma(\frak{s})$, 
\begin{align*}
    I_{\varphi}^3(\phi(x))&=u_{\varphi}(x)I_{\varphi}(\phi(x))\\
    &=I_{\varphi}(u_{\varphi}(x)\phi(x))\\
    &=\iota^{\ast}\circ \varphi (\overline{u_{\varphi}(x)}c(\phi(x)))\\
    &=\iota^{\ast}(\overline{u_{\varphi}(x)}\varphi \circ c(\phi(x)))\\
    &=\overline{u_{\varphi}(\iota^{-1}(x))}I_{\varphi}(\phi(x)).
\end{align*}
Thus we have $u_{\varphi}(\iota^{-1}(x))=\overline{u_{\varphi}(x)}$. 
\end{proof}
Now we prove the following proposition. 
\begin{lem}\label{existenceofI}
There exists a gauge transformation $u$ such that $(u I_{\varphi})^2=\pm 1$. %Moreover, if another gauge transformation $u'$ satisfies $(u' I_{\varphi})^2=1$ or $-1$, we have $(u I_{\varphi})^2=(u' I_{\varphi})^2$. 
\end{lem}
\begin{proof}
From~\cref{uphi2} and the assumption that $H^1(W, \Z)^{-\iota^{\ast}}=0$, we have that there exists a real valued smooth function $f$ such that $f(\iota^{-1}(x))=-f(x)$ and $u_{\varphi}(x)=\pm \exp(if(x))$. We set $u(x)=\exp(-if(x)/2)$. Then we have 
\begin{align*}
    u I_{\varphi} \circ u I_{\varphi}(\phi(x))
    &=u(x) I_{\varphi}(u(x)I_{\varphi}(\phi(x)))\\
    &=u(x)\iota^{\ast}(\overline{u(x)}\varphi \circ c (\phi(x)))\\
    &=u(x)\overline{u(\iota^{-1}(x))}I_{\varphi}^2(\phi(x))\\
    &=u(x)\overline{u(\iota^{-1}(x))}u_{\varphi}(x). 
\end{align*}
From the definition of $u$, we see  $\overline{u(\iota^{-1}(x))}=u(x)$ and $u(x)^2=\pm \overline{u_{\varphi}(x)}$. Thus we have $(uI_{\varphi})^2=\pm 1$.
\end{proof}

%\begin{defi}
%We call $I$ the anti-linear invartible map $uI_{\varphi}$ appeared in the~\cref{existenceofI}. 
%\end{defi}
From the following lemma, we see that an anti-linear involution on $S$ which satisfies some conditions is unique up to $\iota$ invariant gauge transformations. 
\begin{lem}
\label{uniqueI}
Let $I_1, I_2 \colon \frak{s} \to \frak{s}$ be an anti-linear map which satisfies that $I_1, I_2$ covers $\iota$ and $I_1, I_2$ is compatible with Clifford multiplication $\rho$ in the following sense:
\[I_i(\rho(X)\phi(x))=\rho(d\iota(X))I_i(\phi(x)). \; (i=1, 2)\] 
Then there exists a gauge transformation $u_0$ such that $I_2=u_0 I_1$. 
If $I_1$ and the gauge transformation $u_0$ satisfies that $I_1^2=\pm 1$ and $(u_0 I_1)^2=\pm I_1^2$, then $u_0$ is an $\iota$ invariant gauge transformation and $(u_0 I_1)^2=I_1^2$. 
\end{lem}
\begin{proof}
From the assumption of the compatibility of the Clifford multiplication of $I_1, I_2$, we see that $I_1 \circ I_2^{-1}$ is an automorphism of the spin$^c$ structure $\frak{s}$. Thus this is a gauge transformation $u_0$. 

Let us show the second half of the \lcnamecref{uniqueI}. Note that $I_1(\phi)(\iota^{-1}(x))=I_1(\phi(x))$ for all $\phi \in \Gamma(\frak{s})$. We have
\[
    u_0(x)I_1(u_0(\iota^{-1}(x))I_1(\phi(x)))=u_0(x)\overline{u_0(\iota^{-1}(x))}I_1^2(\phi(x))=\pm I_1^2(\phi(x)). 
\]
Thus $u_0$ satsfies that $u_0(x)\overline{u_0(\iota^{-1}(x))}=\pm 1$. From the assumption that $W^{\iota} \neq \emptyset$, we have $u_0(x)\overline{u_0(\iota^{-1}(x))}=1$. 
\end{proof}

%\begin{lem}
%If $\varphi_1$ and $\varphi_2$ are isomorphisms of $\mathrm{spin}^c$ strucutres from $\bar{\frak{s}}$ to  $\iota^{\ast}\frak{s}$. 
%Then there exist a gauge transformation $u_{1, 2} \colon W \to U(1)$ such that $I_{\varphi_2}=u_{1, 2} I_{\varphi_1}$. 
%\end{lem}
From~\cref{uniqueI}, we see that the sign $(u I_{\varphi})^2=\pm 1$ does not depend on the choice of $\varphi$ and $u$. Moreover, If we have $(u' I_{\varphi'})^2=\pm 1$ for another choice of a gauge transformation $u'$ and an isomorphism $\varphi'\colon \bar{\fraks} \to \iota^{\ast}\fraks$, we have that $u I_{\varphi}$ coincides with $u' I_{\varphi'}$ up to some $\iota$-invariant gauge transformation. For abbreviation, we write $u I_{\varphi}$ for $I$. 

\begin{defi}
We say that an involution $\iota$ is of {\it odd type} for $\fraks$ when $I^2=1$. 
\end{defi}

\begin{lem}
\label{lem :codimension 2}
Suppose that $W^{\iota} \neq \emptyset$.
Then $W^{\iota}$ is of codimension-2 if and only if $\iota$ is an odd type.
\end{lem}
\begin{proof}
Let $x_0 \in W$ be an $\iota$ fixed point. Let $U(x_0)$ an $\iota$ invariant normal coordinate chart centered at a point $x_0$. In the open set $U(x_0)$, there is the unique spin structure $\frak{s}_0$. We fix the trivialization $S|_{U(x_0)}\cong U(x_0) \times (\mathbb{H} \oplus \mathbb{H})$ of the spinor bundle of $\frak{s}_0$ 
which satisfies that the Clifford action $\rho$ is represented by $4 \times 4$ matrices as follows:
\[
    \rho(e_0)=\begin{pmatrix}0&-1\\1&0\end{pmatrix}, \;
    \rho(e_1)=\begin{pmatrix}0&i\\i&0\end{pmatrix},\;
    \rho(e_1)=\begin{pmatrix}0&j\\j&0\end{pmatrix},\;
    \rho(e_1)=\begin{pmatrix}0&k\\k&0\end{pmatrix}.
\] 
(In our convantion, $ijk=1$. )
In this chart, the restrictions of the spin$^c$ structures $\frak{s}, \bar{\frak{s}}$ and $\iota^{\ast}\frak{s}$ are isomorphic to the spin$^c$ structure $\frak{s}_0$. Using this identification, we fix an isomorphism of the spin$^c$ structure $\varphi \colon  \bar{\fraks} \to \iota^{\ast}\fraks$. 
Let $\cdot j \colon \frak{s}_0 \to \frak{s}_0$ is an anti-linear map given by the right multiplication of $j$. Then we see that 
$\tilde{\iota}=\iota^{\ast} \circ \varphi \circ c \circ (\cdot j) \colon \frak{s}_0 \to \frak{s}_0$ is a lift of the involution $\iota$ to the spin$^c$ structure $\frak{s}_0$. 
Therefore $\tilde{\iota}_{x_0}$  can be written by $[g, u] \in \mathrm{Spin}^c(4) \cong (\mathrm{Spin}(4) \times U(1))/\{(1, 1), (-1, -1)\}$ where $g$ is a lift of $d\iota_{x_0}$ to $\mathrm{Spin}(4)$. 
Thus we see that for $\phi \in S_{x_0} \cong \mathbb{H} \oplus \mathbb{H}$, 
\[
    I_{x_0}(\phi)=g \phi j u^{-1}
\]
and we have $I^2(\phi)=-g^2\phi$. Hence $\iota$ is odd if and only if $g^2=-1$. One can easily check that $g^2=-1$ if and only if fixed point sets of $\iota$ is codimension $2$.  
\end{proof}

From \cref{lem :codimension 2}, we may obtain a convenient sufficient condition that $\iota$ is an odd type:
\begin{lem}
\label{lem: odd sign half}
Let $W$ be a closed, oriented smooth 4-manifold. Let $\iota : W \to W$ be an orientation-preserving smooth involution.
Let $\fraks$ be a spin$^{c}$ structure on $W$ and suppose that $\iota^{\ast}\fraks=\bar{\fraks}$.
If $\sigma_{\iota}(W) \neq \sigma(W)/2$, then $\iota$ is an odd type.
\end{lem}

\begin{proof}
Assume that $\iota$ is an even type.
By the $G$-signature theorem, $\sigma_{\iota}(W)$ can be obtained by adding $\sigma(W)/2$ to contributions from $W^{\iota}$.
However, by \cref{lem :codimension 2}, $W^{\iota}$ consists only of isolated points, and the contribution from isolated fixed points are zero for a general involution. 
\end{proof}

As well as for involutions on 4-manifolds discussed until here, we can repeat similar arguments for involutions on 3-manifolds.
We summarize it below:

\begin{theo}\label{3diminvolutionI}
Let $Y$ be a closed, oriented three-manifold and $\frak{t}$ be a spin$^c$ structure on $Y$. Let $\iota \colon Y \to Y$ be an involution such that $Y^{\iota} \neq \emptyset$, $\iota^{\ast}\frak{t} \cong \bar{\frak{t}}$, and $H^1(Y, \Z)^{-\iota^{\ast}}=0$. 
Then we have an anti-linear map $I \colon \frak{t} \to \frak{t}$ such that $I$ covers $\iota$ and $I$ is compatible with Clifford multiplication $\rho$ in the following sense:
\[I(\rho(X)\phi(x))=\rho(d\iota(X))I(\phi(x)).\]
The choice of $I$ is unique up to $\iota$ invariant gauge transformations. 
Moreover, $I$ satisfies that $I^2=\pm 1$ and $I^2=1$ if and only if the fixed point set of $\iota$ is codimension $2$. 
\end{theo}

From now, we assume that $\iota$ is of odd type.

Next, we show that the Seiberg--Witten equation is equivariant with the involution $-\iota^* \oplus I$. 

\begin{prop}\label{sw is equiv I}
Let us define an involution on $\Omega^1(W) \oplus \Gamma(S^+)$ and $\Omega^+(W) \oplus \Gamma(S^-)$ by $-\iota^* \oplus I$. We have the Seiberg--Witten equations with $\iota$-invariant Riemannian metric on $W$ is equivariant with this involution. 
\end{prop}
\begin{proof}
Let $A'_0$ be a spin$^c$ connection on $\mathfrak{s}$. We define the covariant derivative $d_{A_0} \colon \Omega(W) \otimes \Gamma(S) \to \Omega(W) \otimes \Gamma(S)$ to be
\[
d_{A_0}:=\frac{1}{2}(d_{A'_0}+(\iota^* \otimes I)\circ d_{A'_0}\circ (\iota^* \otimes I)). 
\]
It is easy to check that this is an $\iota^*\otimes I$ invariant spin$^c$ connection on $\mathfrak{s}$. Moreover, we have that the curvature form $F_{A_0^{\tau}}$ of $\det{S^+}$ induced by the connection $A_0$ satisfies that $\iota^*F_{A_0^{\tau}}=-F_{A_0^{\tau}}$ since $F_{A_0}$ is an imaginary-valued $2$-form and $I$ is an anti-linear. 
We set $A_0$ to be a reference connection. Let $A=A_0+\sqrt{-1}a$. 
We only need to show that the non-linear terms of the Seiberg--Witten equations are equivariant with the involution. From \cref{IandCliffford},  we have
\begin{equation}\label{eq: nonlinear1}
    I(\sqrt{-1}\rho(a)\phi)=-\sqrt{-1}I(\rho(a)\phi)= -\sqrt{-1}\rho(\iota^*a)I(\phi) 
\end{equation}
for any $a \in \Omega^*(W)$ and $\phi \in \Gamma(S^+)$. Thus we have that the Dirac equation is equivariant to the involution. 

Next, we show that the equation of curvature
\begin{equation*}
    F_{A^\tau}^+=\sqrt{-1}d^+a + F_{A_0^\tau}^+=-\sqrt{-1}\tau(\phi, \phi)
\end{equation*}
is equivariant under the involution. 
The quadratic form $\tau \colon S^+ \otimes S^+ \to \Lambda^+$ is characterized by the following relation:
\[
\langle b, \tau(\phi_1, \phi_2)\rangle_{\Lambda^+}=-\langle\sqrt{-1}\rho(b)\phi_1 , \phi_2\rangle_{S^+}
\]
where $b \in \Omega^+(W)$ and $\phi_1, \phi_2 \in \Gamma(S^+)$. 
The inner product $\langle, \rangle_{\Lambda^+}$ and $\langle, \rangle_{S^+}$ are invariant under the involution $\iota^*$ and $I$ respectively. From \eqref{eq: nonlinear1}, we have
\begin{align*}
\langle b, \iota^*\tau(\phi_1, \phi_2)\rangle_{\Lambda^+}
&=\langle \iota^*b, \tau(\phi_1, \phi_2)\rangle_{\Lambda^+}\\
&=-\langle\sqrt{-1}\rho(\iota^*b)\phi_1 , \phi_2\rangle_{S^+}\\
&=-\langle I(\sqrt{-1}\rho(\iota^*b)\phi_1) , I(\phi_2)\rangle_{S^+}\\
&=\langle\sqrt{-1}\rho(b)I(\phi_1) , I(\phi_2)\rangle_{S^+}\\
&=-\langle b, \tau(I(\phi_1), I(\phi_2))\rangle_{\Lambda^+}. 
\end{align*}
This completes the proof. 
\end{proof}

We can prove that the Seiberg--Witten equation on $3$-manifolds with involution $\iota$ is invariant under the involution $-\iota^* \oplus I$ in a similar way. 

\section{Floer homotopy type of rational homology 3-spheres with involution and Fr\o yshov type invariants}
\label{section Floer homotopy type of rational homology 3-spheres with involution and Fro yshov type invariants}
In \cite{KMT}, we defined a ($\Z_4$-equivariant) ``doubled" Seiberg--Witten Floer stable homotopy type of spin rational homology 3-spheres with involution
 $DSWF_{\Z_4}(Y, \frakt, \iota)$.
The doubling construction was considered mainly to define a $K$-theoretic Fr{\o}yshov-type invariant easily.
In this \lcnamecref{section Floer homotopy type of rational homology 3-spheres with involution and Fro yshov type invariants}, we define Seiberg--Witten Floer stable homotopy type of real spin$^c$ rational homology 3-spheres
 $SWF_{\Z_2}(Y, \frakt, \iota)$, without taking double, and define a Fr{\o}yshov type invariant
\[
\delta_R(Y, \frakt, \iota) \in \frac{1}{16}\Z
\]
applying the $\Z_2$-equivariant ordinary cohomology to $SWF_{\Z_2}(Y, \frakt, \iota)$.
Moreover, when $\mathfrak{t}$ is spin, we also define 
two Fr{\o}yshov type invariants
\[
\underline{\delta}_R(Y, \frakt, \iota), \bar{\delta}_R(Y, \frakt, \iota) \in \frac{1}{16}\Z 
\]
applying $\Z_4$-equivariant ordinary cohomology to $SWF_{\Z_4}(Y, \frakt, \iota)$.
These invariants are analogues of Stoffregen's invariants $\underline{\delta}_R$ and $\bar{\delta}_R$. 
Throughout this \lcnamecref{section Floer homotopy type of rational homology 3-spheres with involution and Fro yshov type invariants},
we consider cohomologies with coefficient $\mathbb{F}=\Z_2$.

\subsection{Representations}\label{representations}

First let us consider $\Z_2$-representations.
Let $\R$ be the trivial real 1-dimensional representation of $\Z_2$, and $\C$ denote the complex 1-dimensional representation of $\Z_2$ defined as the scalar multiplication of $\Z_2=\{1,-1\}$.
For a finite-dimensional vector space $V$, let $V^+$ denote the one-point compactification of $V$.

Recall that we defined the group $G$ to be the cyclic group of order $4$ generated by $j \in \Pin(2)$, i.e. 
\[
G = \{1, j, -1, -j\}.
\]
Define a subgroup $H$ of $G$ by
\[
H = \{1,-1\} \subset G.
\]
Let $\R$ denote the trivial 1-dimensional real representation of $G$.
Let $\tilde{\R}$ be the 1-dimensional real representation space of $G$ defined by the surjection $G \to \Z_2=\{1,-1\}$ and the scalar multiplication of $\Z_2$ on $\R$.

Let $\tilde{\C}$ be a 1-dimensional complex representation of $G$ defined via the surjection $G \to \Z_2$ and the scalar multiplication of $\Z_2$ on $\C$. 
Note that, for an even natural number $s$, say $2t$, there is an isomorphism of real representations $\tilde{\R}^{2s} \cong \tilde{\C}^t$.
We introduce also a $G$-representation $\C$ (the same notation of the complex number) which is the complex $1$-dimensional representation defined by assigning $j \in G$ to $i$ in $\C$.

\subsection{Numerical invariants $d$, $\underline{d}$, $\overline{d}$}

We first define the main ingredient of the Fr{\o}yshov invariant using $\Z_2$-equivariant cohomology,
following Stoffregen's formulation~\cite{Sto171}.

\begin{defi}
Let $G$ be a group and $H$ be a subgroup of $G$.
Let $\mathcal{V}$ be a coutable direct sum of a fixed 1-dimensional real representation $G$.
Let $X$ be a pointed finite $G$-CW complex.
We call $X$ a {\it space of type $(G, H)$-SWF} if 
\begin{itemize}
    \item $X^{H}$ is $G$-homotopy equivalent to $V^+$, where $V$ is a finite dimensional subspace of $\mathcal{V}$.
    \item $H$ acts freely on $X \setminus X^{H}$.
\end{itemize}
The dimension $\dim V$ is called the {\it level} of $X$.  We put $\mu (X) \in  \Q/2\Z$ by $\mu (X) = {\dim V}/{2}\ \operatorname{mod}\  2$.

\end{defi}

The situation we have in mind is either
\[
(G,H) = (\Z_2, \Z_2) \text{ or } (\Z_4, \Z_2).
\]
We often drop $H$ in the notation in $(G,H)$.

First, let us consider $(G,H) = (\Z_2, \Z_2)$. 
Note that 
\[
\tilde{H}_{\Z_2}^\ast(S^0) \cong \Z_2[W],
\]
where $W$ is of degree 1.
For a space $X$ of type $\Z_2$-SWF,
define
\begin{align*}
d(X) &= \min\Set{m \geq 0| \exists x \in \tilde{H}_{\Z_2}^m(X),\ W^lx\neq0\ (\forall l \geq 0) }\\
&= \min\Set{m \geq 0| \exists x \in \tilde{H}_{\Z_2}^m(X), 0\neq \iota^* x \in  \wt{H}^*_G(X^H)  }.
\end{align*}
(See \cite[Equation (19)]{Sto171}.)
By an equivariant localization theorem (see, e.g., \cite[Theorem 2.3]{Sto171}),
we have $d(X) < +\infty$.

Next, let us consider $(G,H) = (\Z_4, \Z_2)$. 
In this case, we have 
\[
\wt{H}_{G}^\ast (S^0) \cong \Z_2 [U, Q]/(Q^2 = 0),
\]
where $\deg U=2$ and $\deg Q=1$.
For a space $X$ of type $\Z_4$-SWF,
define
\[
\overline{d}(X) = \min\Set{m \equiv 2 \mu (X) \geq 0\ (\mathrm{mod}\ 2) | \exists x \in \tilde{H}_{\Z_4}^m(X),\ U^lx\neq0\ (\forall l \geq 0) }
\]
and
\[
\underline{d}(X) = \min\Set{m \equiv 2 \mu (X) +1 \geq 0\ (\mathrm{mod}\ 2)| \exists x \in \tilde{H}_{\Z_4}^m(X),\ U^lx\neq0\ (\forall l \geq 0) }-1.
\]

Again, by an equivariant localization theorem (see, e.g., \cite[Theorem 2.3]{Sto171}), we have that $\overline{d}(X), \underline{d}(X) < \infty$.

We also use alternative descriptions of invariants $\underline{d}$ and $\overline{d}$ by using {\it infinity version} of equivariant cohomology. For a space $X$ of type $G$-SWF, we define 
\[
\infty \wt{H}^*_G (X) := \begin{cases}
    \operatorname{Im} (\wt{H}_G(X) \to  U^{-1} \wt{H}_G(X)) \text{ if } G=\Z_4\\
     \operatorname{Im} (\wt{H}_G(X) \to  W^{-1} \wt{H}_G(X))  \text{ if } G=\Z_2. 
\end{cases}
\]

We have the following classification result of ideals of $U^{-1} \Z_2[U, Q]/(Q^2=0)$. 
\begin{lem}
    Any graded ideal $\mathcal{J}$ of $U^{-1} \Z_2[U, Q]/(Q^2=0)$ such that $U^{-1} \mathcal{J} =U^{-1} \Z_2[U, Q]/(Q^2=0)$ 
    have the following form: 
    \[
    \mathcal{J} = (U^i, Q U^j)
    \]
    for some $i  \geq j \geq 0 $.
\end{lem}
\begin{proof}
The proof is essentially the same as the proof of \cite[Lemma 2.8]{Ma16}. 
\end{proof}
 If $\infty \wt{H}^{n+*}_G (X) = (U^i, Q U^j)$, then we have 
 \begin{align}\label{infinity descp coh}
 \underline{d} (X) = i+ n \text{ and }  \overline{d} (X) = j+ n.
  \end{align}
 In particular, this expression \eqref{infinity descp coh} enables us to check 
 \[
 \underline{d} (X) \leq \overline{d} (X). 
 \]

 Moreover, we have the following properties with respect to stabilizations: 
 \begin{lem}
      Let $X$ be a space of type $G=\Z_4$ SWF, and $V$ a representation of $G=\Z_4$. 
      Then, we have 
      \begin{align}\label{stabilization d-bar}
      \underline{d}(\Sigma^V X)= \underline{d}(X) + \dim V \text{ and } \overline{d}(\Sigma^V X)= \overline{d}(X) + \dim V. 
      \end{align}
      Similarly, for a space $X$ of type $G=\Z_2$ SWF and a $\Z_2$-representation $V$, we have 
      \begin{align}\label{stabilization of d}
      {d}(\Sigma^V X)= {d}(X) + \dim V. 
       \end{align}
 \end{lem}
 \begin{proof}
     As it is proven in \cite[Proposition 2.2]{Ma16}, for any finite-dimensional representation $V$ of $G=\Z_4$ or $\Z_2$, we have the suspension isomorphism of $\wt{H}_G(S^0)$-modules:
\begin{align}\label{suspension isom}
     \wt{H}^*_G (\Sigma^V X ;\Z_2) \cong \wt{H}^{*-\dim V}_G ( X ;\Z_2) . 
       \end{align}
The equations \eqref{stabilization d-bar} and \eqref{stabilization of d} follow from \eqref{suspension isom}. 
 \end{proof}

Later we shall use the following elementary lemmas.
Recall that, given a $G$-vector bundle $E \to B$, the (mod 2) $G$-equivariant Euler class $e_G(E) \in H_G^\ast(B)$ is defined by
\[
e_G(E) = e(EG \times_G E \to EG \times_G B) \in H_G^\ast(B).
\]

\begin{lem}
\label{lem: equiv Euler Z2}
The cohomology class $W \in \tilde{H}_{\Z_2}^\ast(S^0)$ coincides with $e_{\Z_2}(\tilde{\R})$, where $\tilde{\R}$ is regarded as a $\Z_2$-vector bundle over $\{\pt\}$.
\end{lem}

\begin{proof}
By definition, we have
\[
e_{\Z_2}(\tilde{\R})
=
e(E\Z_2 \times_{\Z_2} \tilde{\R} \to B\Z_2) \in \tilde{H}^1_{\Z_2}(S^0).
\]
The right-hand side coincides with $w_1(E\Z_2 \times_{\Z_2} \tilde{\R})$ and it is the generator of $H^1(B\Z_2) = \tilde{H}^1_{\Z_2}(S^0)$.
This completes the proof.
\end{proof}

\begin{lem}
\label{lem: comparison cohomology}
The cohomology class $Q \in \tilde{H}_{\Z_4}^\ast(S^0)$ is the image of $W \in \tilde{H}_{\Z_2}^\ast(S^0)$ under the natural map $\tilde{H}^\ast_{\Z_2}(S^0) \to \tilde{H}^\ast_{\Z_4}(S^0)$ induced from the surjection $\Z_4 \to \Z_2$.
\end{lem}

\begin{proof}
The surjection $\Z_4 \to \Z_2$ induces a surjection $\tilde{H}_1(B\Z_4) \to \tilde{H}_1(B\Z_2)$.
Passing to the dual, this induces the isomorphism of the cohomologies $\tilde{H}^1(B\Z_2) \to \tilde{H}^1(B\Z_4)$, thus the generator $W$ of $\tilde{H}^1(B\Z_2)$ maps to the generator $Q$ of $\tilde{H}^1(B\Z_4)$.
\end{proof}

For $S^0$, the simplest example of the space of type SWF, it is easy to see that the above invariants coincide,
\[
d(S^0) = \underline{d}(S^0) = \overline{d}(S^0)=0.
\]
Below we exhibit an example for which the invariants 
$d$ and $\underline{d}$, $\overline{d}$ are distinct, following \cite[Example~2.10]{Ma16}.

\begin{ex}
\label{ex: easiest non-trvial ex}
Set $G=\Z_4$, and
let $\tilde{G}$ denote the unreducible suspension of $G$.
We regard $\tilde{G}$ as a based space by choosing one of the cone point of $\tilde{G}$ as the base point.
This space $\tilde{G}$ is obviously a space of type $\Z_4$-SWF.
We claim that
\begin{align}
\label{eq: easiest non-trivial ex of ds}
\underline{d}(\tilde{G})=0,\quad
d(\tilde{G})=1,\quad
\overline{d}(\tilde{G})=2.
\end{align}

To see this, first note that the $\Z_2$-invariant part of $\tilde{G}$ is given by $\tilde{G}^{\Z_2} = S^0$.
The cone of the inclusion map $S^0 \hookrightarrow \tilde{G}$ is given by the reduced suspension $\Sigma^{\R}G_+$ of $G_+=G \sqcup \mathrm{pt}$.
Thus we obtain the long exact sequence
\begin{align}
\label{eq: long ex tilde G}
\cdots\to 
H_G^\ast(\Sigma^{\R}G_+)
\to 
\tilde{H}_G^\ast(\tilde{G})
\to
{H}_G^\ast(S^0)
\to 
H_G^{\ast+1}(\Sigma^{\R}G_+)
\to \cdots.
\end{align}
It follows from this long exact sequence combined with
\[
\tilde{H}_G^{\ast+1}(\Sigma^{\R}G_+) \cong \tilde{H}_G^{\ast}(G_+)
\cong H^\ast(\mathrm{pt})
\]
that the restriction map
$\tilde{H}_G^m(\tilde{G}) \to \tilde{H}_G^m(S^0)$
is isomorphic if $m\geq2$.
On the other hand, it is clear that $\tilde{H}_G^0(\tilde{G})=0$.
This combined with the above long exact sequence \eqref{eq: long ex tilde G} implies that the restriction $\tilde{H}_G^1(\tilde{G}) \to \tilde{H}_G^1(S^0)$ is isomorphic.
The claims that $\underline{d}(\tilde{G})=0$ and that $\overline{d}(\tilde{G})=2$ immediately follow from these computations.

Next, we prove the claim $d(\tilde{G})=1$.
First, we obviously have $\tilde{H}_{\Z_2}^0(\tilde{G})=0$.
Second, by combining a long exact sequence for $\tilde{H}_{\Z_2}^\ast$ analogous to \eqref{eq: long ex tilde G} with that
\[
\tilde{H}_{\Z_2}^{2}(\Sigma^{\R}G_+) \cong \tilde{H}_{\Z_2}^{1}(G_+)
\cong H^1(\mathrm{pt} \sqcup \mathrm{pt})=0,
\]
we can see that the restriction map $\tilde{H}^1_{\Z_2}(\tilde{G}) \to \tilde{H}^1_{\Z_2}(S^0)$ is surjective.
The claim $d(\tilde{G})=1$ directly follows from these observations.
\end{ex}

\begin{comment}
Define 
\[
\operatorname{Tor}(X) := \max \{ n |  \exists x \in \wt{H}^*_{\Z_4} (X) , \   U^n x \neq 0, U^m x = 0 \text{ for some } m \}. 
\]

\begin{prop}For a space $X$ of type $(\Z_4, \Z_2)$-SWF, the following inequality holds: 
\[
\max\{ |d(X)- \underline{d}(X) | , |d(X)- \overline{d}(X) |  \} \leq 2 \operatorname{Tor}(X). 
\]
\end{prop}

\begin{prop}For spaces $X_1$ and $X_2$ of type $(\Z_4, \Z_2)$-SWF, one has
\[
    \operatorname{Tor} (X_1 \wedge X_2) \leq \max \{\operatorname{Tor} (X_1), \operatorname{Tor} (X_2)\}. 
    \]
\end{prop}
\end{comment}

Next we see a certain monotonicity for the quantities $d, \underline{d}, \overline{d}$:

\begin{lem}
\label{lem: sourse of Froyshov ineq}
 Let $X$ and $X'$ be spaces of type $\Z_2$-SWF at the same level.
Suppose that there exists a pointed $\Z_2$-equivariant map $f : X \to X'$ whose $\Z_2$-fixed-point set map is a homotopy equivalence.
Then we have
\[
d(X) \leq d(X').
\]
Similarly, for spaces $X, X'$ of type $\Z_4$-SWF at the same level, if there exists a pointed $\Z_4$-equivariant map $f : X \to X'$ whose $\Z_2$-fixed-point set map is a $\Z_4$-homotopy equivalence, then we have
\[
\overline{d}(X) \leq \overline{d}(X'),
\quad \underline{d}(X) \leq \underline{d}(X').
\]
\end{lem}

\begin{proof}
%The proof is analogous to that of \cite[Proposition~2.15]{Ma16}, and we give a sketch of the proof.
This is a standard argument, but we give a sketch of the proof for the reader's convenience.
We first show the claim for the invariant $d$. Set $(G,H)=(\Z_2, \Z_2)$.
Consider the following commutative diagram: 
\[
  \begin{CD}
    X^H @>{f^H}>>  (X')^H \\
  @V{}VV    @VVV \\
    X   @>{f}>>  X'.
  \end{CD}
\]
Here the vertical maps are inclusions.
By applying $\wt{H}^*_G$, one has a commutative diagram 
\[
  \begin{CD}
    \wt{H}^*_G(X^H) @<{(f^H )^*}<< \wt{H}^*_G( (X')^H) \\
  @AAA    @AAA \\
    \wt{H}^*_G(  X )  @<{f^*}<<  \wt{H}^*_G (X').
  \end{CD}
\]
 
It follows from the assumption that $(f^H)^*$ is a homotopy equivalence combined with the equivariant localization theorem that
$f^\ast$ is an isomorphism in large enough degree.
From this combined with that $f^\ast$ commutes with the $W$-action, the commutative diagram above implies that
\begin{align}
\label{eq: pullback}
f^\ast\left(\Set{x' \in \tilde{H}_{G}^\ast(X') | W^lx'\neq0\ (\forall l \geq 0) }\right)
\subset \Set{x \in \tilde{H}_{G}^\ast(X) | W^lx\neq0\ (\forall l \geq 0) }.
\end{align}
The desired inequality $d(X) \leq d(X')$ follows from this.

For the invariants $\overline{d}, \underline{d}$, we may repeat the above argument for $(G,H)=(\Z_4, \Z_2)$, with replacing $W$ with $U$ in \eqref{eq: pullback}.
\end{proof}

We also note a smash formula of $d$: 
\begin{lem}\label{conn sum for d}
Let $X$ and $X'$ be spaces of type $\Z_2$-SWF at the level $l$ and $ l'$.
Then, $X \wedge X'$ becomes a space of type $\Z_2$-SWF at the level $l+l'$ and 
\[
d(X\wedge X' ) \geq   d(X) + d(X') \]
holds. 
\end{lem}
\begin{proof}
Set $(G,H)=(\Z_2,\Z_2)$.
The following commutative diagram is a key ingredient in the proof:  
\[
  \begin{CD}
    \wt{H}^*_G((X\wedge X')^H) @>{\cong_{\#^H}}>> \wt{H}^*_G(X^H) \otimes \wt{H}^*_G ((X')^H ) \\
  @A{(\iota \wedge \iota')^\ast }AA    @A{\iota^\ast \otimes \iota'^\ast}AA \\
    \wt{H}^*_G(  X\wedge X' )  @>{\#}>>  \wt{H}^*_G (X) \otimes \wt{H}^*_G( X'),
  \end{CD}
\]
where the horizontal maps are induced from the inclusion maps $X \to X \wedge X'$ and $X' \to X \wedge X'$ and the vertical maps are induced from inclusions $X^H \to X$ and $(X')^H \to X'$. 
\begin{comment}
Put 
$
d(X) = m $ and $d(X') = m'$ 
such that there are $
x \in \tilde{H}_{\Z_2}^m(X)$with $ \iota^* x\neq0$
and
$ x' \in \tilde{H}_{\Z_2}^m(X')$ with $(\iota')^* x'\neq0.$
Since we have the natural map $\wt{H}^*_G (X) \otimes \wt{H}^*_G( X')\to \wt{H}^*_G (X\wedge X')$, we consider $x^\otimes :=\# ( x \otimes x') \in \wt{H}^{m+m'}_{\Z_2} (X\wedge X')$. Then we have $(\iota^\otimes )^* x^\otimes  \neq 0$ by commutativity of diagram. Thus, we conclude 
\[
d(X\wedge X') \leq m + m'. 
\]
\end{comment}
%Take an element $x \in \wt{H}^n_{G}(X \wedge X')$ so that $\iota^\otimes x \neq 0$ and $d(X \wedge X')=n$. Then, for an expression  $\iota \circ  \# (x) = \sum W^{j_i} a_i\otimes b_i$, $W^{j_1}a_1$ and $b_1$ are non-zero elements of the images of $\wt{H}^* _G (X)$ and $\wt{H}^* _G (X')$ under the restriction maps. Then, by the definition of $d$, one has $d(X) \leq \deg (W^{j_1}a_1)$ and $d(X') \leq \deg (b_1)$. On the other hand, we have $\deg (W^{j_1}a_1)+  \deg (b_1)= \deg (x) =n$. This completes the proof. 
Take an element $x \in \wt{H}^n_{G}(X \wedge X')$ so that $\iota^\otimes x \neq 0$ and $d(X \wedge X')=n$. Then, for an expression 
$(\iota \otimes \iota')^\ast \circ  \# (x) = \sum_i a_i\otimes b_i$ with $a_i \in \wt{H}^* _G (X^H) \setminus \{0\}$ and $b_i \in \wt{H}^* _G ((X')^H)\setminus \{0\}$, each of $a_1$ and $b_1$ are images of elements of $\wt{H}^* _G (X)$ and $\wt{H}^* _G (X')$ under the restriction maps, respectively.
Then, by the definition of $d$, one has $d(X) \leq \deg (a_1)$ and $d(X') \leq \deg (b_1)$.
On the other hand, we have $\deg (a_1)+  \deg (b_1)= \deg (x) =n$. This completes the proof.
\end{proof}

At the end of this subsection, we provide a key ingredient of a Theorem B type theorem in our theory. 

\begin{lem}
\label{Theorem B oiginal}
For spaces $X, X'$ of type $\Z_4$-SWF, if there exists a pointed $\Z_4$-equivariant map $f : X \to X'$ whose $\Z_2$-fixed-point set map is induced from a $\Z_4$-injective linear map whose image is of codimension-1, then we have
\[
\quad \underline{d}(X) \leq \overline{d}(X').
\]
\end{lem}

\begin{proof}
Set $(G,H) = (\Z_4, \Z_2)$.
Let $s, s'$ denote the levels of $X$ and $X'$ respectively.
By the assumption on $f^H$, we have that $s+1=s'$.
Recall that the $\Z_4$-equivariant Thom isomorphism implies that $\wt{H}^*_G(X^H) \cong \wt{H}^*_G(S^0)_{[s]}$, where $[s]$ denotes the degree $s$-shift.
Let $[s] : \wt{H}^*_G(S^0) \to \wt{H}^*_G(S^0)_{[s]}$ denote the degree $s$-map defined by just the degree shift.

As in the proof of \cref{lem: sourse of Froyshov ineq}, consider the commutative diagram
\begin{align}
\label{eq: diag shift}
\begin{split}
\xymatrix{
\wt{H}^*_G(S^0) \ar[d]_{[s]} & \wt{H}^*_G(S^0) \ar[d]_{[s']}\\
\wt{H}^*_G(X^H) & \wt{H}^*_G((X')^H) \ar[l]_{(f^H)^\ast} \\
\wt{H}^*_G(X)\ar[u]^{i^\ast} & \wt{H}^*_G(X'), \ar[l]_{f^\ast} \ar[u]^{(i')^\ast}
}
\end{split}
\end{align}
where $i : X^H \to X$ and $i' : (X')^H \to X'$ are inclusions.

We claim that 
\[
[-s] \circ (f^H)^\ast \circ [s'] : \wt{H}^*_G(S^0) \to \wt{H}^*_G(S^0)
\]
is given by multiplication by $Q$.
To see this, note that $\wt{H}^*_G(X^H)$ and $\wt{H}^*_G(X^H)$ are rank 1 free $\wt{H}^*_G(S^0)$-modules generated by the $G$-equivariant mod 2 Thom classes $\tau_G$ and $\tau_G'$ of $\tilde{\R}^s \to \{\mathrm{pt}\}$ and $\tilde{\R}^{s'} \to \{\mathrm{pt}\}$, respectively. 
By the assumption on $f^H$, we have 
\[
(f^H)^\ast \tau_G'
= e_G(\tilde{\R})\tau_G,
\]
where $e_G(\tilde{\R})$ is the $G$-equivariant mod 2 Euler class of $\tilde{\R} \to \{\mathrm{pt}\}$.
Since $G$ acts on $\tilde{\R}$ as $\{\pm1\}$-multiplication, we have that $e_G(\tilde{\R}) = Q$ from \cref{lem: equiv Euler Z2,lem: comparison cohomology}.
This proves the claim.

In each of $\wt{H}^*_G(X^H)$ and $\wt{H}^*_G((X')^H)$, there are exactly two $U$-towers that are degree shifts of $\{U^l\}_{l \geq 0}$ and $\{QU^l\}_{l \geq 0}$ in $\wt{H}^*_G(S^0)$.
For $i=0,1$, let $\mathcal{T}_i \subset \wt{H}^*_G(X^H)$ and $\mathcal{T}_i'\subset \wt{H}^*_G((X')^H)$ denote these towers such that every $x \in \mathcal{T}_i$ and $x' \in \mathcal{T}_i'$ satisfy that 
\[
\deg{x} \equiv 2\mu(X)+i\quad \text{and} \quad \deg{x'} \equiv 2\mu(X')+i \mod 2.
\]
It follows from the above claim that the restriction of $(f^H)^\ast$ on $\mathcal{T}_0'$ gives a bijection onto $\mathcal{T}_1$. 
This combined with the commutative diagram \eqref{eq: diag shift} implies that $\underline{d}(X) \leq \overline{d}(X')$.
\end{proof}

\subsection{Duality}
\label{subsec Duality}

We prove a duality theorem for $d$, $\overline{d}$, and $\underline{d}$, which can be seen as an analog of \cite[Proposition 2.13]{Ma16}. 
In order to prove the duality, we use the following duality relations among equivariant cohomologies.
See \cite[Section XVI. 8]{May96}, \cite[Subsection 2.2]{Ma16} for the definition of (equivariant) $V$-dual.

If $X$ and $X'$ are $V$-dual with $m = \dim V$, the coBorel cohomology of $X'$ 
can be viewed realized as follow: 
\[
\mathrm{c}\tilde{H}^G_* (X; \Z_2) \cong \tilde{H}_G^{m-* }(X'; \Z_2) . 
\]
%\marginpar{Hokuto: $cH$ or $\mathrm{c}H$? $tH$ or $\mathrm{t}H$? Unify here and the introduction.}

Moreover, there is a long exact sequence relating Borel homology with co-Borel and Tate homologies: 
\[
\cdots \to \mathrm{c}\tilde{H}^G_* (X; \Z_2) \to \mathrm{t}\tilde{H}^G_* (X; \Z_2) \to  \tilde{H}^G_{*-1} (X; \Z_2) \to \cdots.
\]

For our purpose, it is also convenient to describe $d$, $\overline{d}$, and $\underline{d}$ in terms of homologies as in \cite[(17), (18), and (19)]{Ma16}. 
We first define 
\[
\infty \wt{H}_*^G (X) := \bigcap_{m \geq 0} \operatorname{Im} U^m : \wt{H}_*^G (X) \to \wt{H}_{*-2m}^G (X). 
\]
Then, alternatively, for a space $X$ of type $G=\Z_4$-SWF of level $s$, we can write 
\begin{align*}
    \underline{d} (X) &=\min  \Set{ r \equiv  s \mod 3  | \exists x , \ 0\neq x \in \infty \wt{H}_r^G (X)},  \\
    \overline{d}(X) &= \min  \Set{ r \equiv  s+1 \mod 3  | \exists x , \ 0\neq x \in \infty \wt{H}_r^G (X)}, \\ 
    d(X) & =  \min  \Set{ r \equiv  s \mod 2  | \exists x , \ 0\neq x \in \infty \wt{H}_r^{\Z_2} (X)}.  
\end{align*}

\begin{lem}\label{duality for d}
Let $G$ be $\Z_2$.
If $X$ and $X'$ are equivariantly $V$-dual, then \[
d(X)= \dim V -  d(X')
\]
hold. 
Moreover, if $G=\Z_4$, then we have 
\[
 \underline{d}(X)= \dim V -  \overline{d}(X'). 
\]
\end{lem}

\begin{proof}
Suppose the level of $X$ is $s$ and $m= \dim V$. Then the level of $X'$ is $m-s$ by duality. 
The proof for $d$ is an easier version of the proof for $\underline{d}$ and $\overline{d}$. Thus, we only write a proof of the second equality. The main strategy is the same as the proof of \cite[Proposition 2.13]{Ma16}. 
We have the exact sequence connecting Borel, coBorel and Tate cohomologies: 
\[
\cdots \to \mathrm{c}\tilde{H}^G_* (X; \Z_2) \to \mathrm{t}\tilde{H}^G_* (X; \Z_2) \to  \tilde{H}^G_{*-1} (X; \Z_2) \to \cdots.
\]
Since $X$ and $X'$ are $V$-dual, from the associated isomorphism 
\[
\mathrm{c}\tilde{H}^G_* (X; \Z_2) \cong \tilde{H}_G^{m-* }(X'; \Z_2),  
\]
we have 
\[
\cdots \to \tilde{H}_G^{m-* }(X'; \Z_2) \to \mathrm{t}\tilde{H}^G_* (X; \Z_2) \to  \tilde{H}^G_{*-1} (X; \Z_2) \to \cdots.
\]
Moreover, the localization theorem implies 
\[
\mathrm{t}\tilde{H}^G_* (X; \Z_2) \cong (U^{-1} \Z_2[U, Q]/ Q^2=0)_{s+2}. 
\]
Then by the same arguemnt given in \cite[Proof of Propsotion 2.13]{Ma16}, after shifting degrees by $s + 2$, the exact sequence above induces 
\[
\cdots \to  \infty \wt{H}^{m-s-2+*}_G (X') \to U^{-1} \Z_2[U, Q]/ (Q^2=0) \to \wt{H}_{*+s}^G (X) \to \cdots  
 \]
 which connects the infinity versions of homology and cohomology directly. Since we have a concrete formula of infinity versions using $\underline{d}$ and $\overline{d}$. This enables us to write 
 \begin{align*}
     \infty \wt{H}^{m-s-2+*}_G (X')  &= \begin{cases}
          \Z_2 \text{ if } * =m-s-2 +  \underline{d} (X') + 2 j \  j \geq 0,\\
          \Z_2 \text{ if } * =m-s-3 +  \overline{d} (X') + 2 j \  j \geq 0,\\
          0 \text{ if otherwise}, 
     \end{cases} \\ 
     \infty \wt{H}^G_{* +s } (X) &= \begin{cases}
       \Z_2 \text{ if } * =-s +  \underline{d} (X) + 2 j \  j \geq 0,\\
          \Z_2 \text{ if } * =-s+1 + \overline{d} (X) + 2 j \  j \geq 0,\\
          0 \text{ if otherwise.}
     \end{cases}
 \end{align*}
 This completes the proof. 
\end{proof}

\subsection{Spectrum classes}
\label{subsection Doubling}

Let $(G, H)$ be either $(\Z_2,\Z_2)$ or $(\Z_4,\Z_2)$, where, in the latter case, $\Z_2$ is regarded as a subgroup of $\Z_4$ in a natural way.
When $(G, H) = (\Z_2,\Z_2)$, we set
\[
\calV = \oplus_{\N} \R,
\quad \calW = \oplus_{\N} \tilde{\R},
\]
and
When $(G, H) = (\Z_4,\Z_2)$, we set
\[
\calV = \oplus_{\N} \tilde{\R},
\quad \calW = \oplus_{\N} \C.
\]

\begin{comment}
\begin{rem}
There are several conventions on the choice of $n$.
\begin{align*}
\text{our current $n$} 
&= 2 \times \text{(Manolescu's $n$ in his first paper)}\\
&= 4 \times \text{(our last $n$)}\\
&= 4 \times \text{(Manolescu's $n$ in the 10/8 paper)}.
\end{align*}
\end{rem}
\end{comment}

Following \cite[Section~4]{Ma14}, 
consider a triple $(X,m,n)$, where 
$X$ is a space of type $G$-SWF, and $m \in \Z$ and $n \in \Q$.
For a finite dimensional representation $V$ of $G$ and a pointed $G$-space $X$, we denote by $\Sigma^V X$ the suspension $V^+ \wedge X$.
Set
\[
\mathbb{K}_G = 
\begin{cases}
\tilde{\R} \text{\quad if \quad} G=\Z_2\\   \C \text{\quad if \quad} G=\Z_4,
\end{cases}\quad
k_G = \dim_{\R}\mathbb{K}_G = 
\begin{cases}
1 \text{\quad if \quad} G=\Z_2\\   2 \text{\quad if \quad} G=\Z_4
\end{cases}
\]
where $\R, \tilde{\R}, \C, \tilde{\C}$ are the representations we introduced in \cref{representations}. 

\begin{defi}
For such triples $(X,m,n), (X',m',n')$,
we say that they are {\it $G$-stably equivalent} to each other if $n-n' \in \Z$ and there exist finite dimensional subspaces $V, V'$ of $\calV$  and $W, W'$ of $\calW$ and a pointed $G$-homotopy equivalence
\[
\Sigma^{V} \Sigma^{W}X
\to \Sigma^{V'} \Sigma^{W'}X',
\]
where $\dim_\R V -\dim_\R V' = m'-m$ and $\dim_{\mathbb{K}_G} W -\dim_{\mathbb{K}_G} W' = n'-n$.
Define $\mathfrak{C}_{G}$ as the set of $G$-stable equivalence classes of triples $(X,m,n)$.
An element of $\mathfrak{C}_{G}$ is called  a {\it spectrum class}.
\end{defi}

Informally, we may think of the triple $(X,m,n)$ as the formal desuspension of $X$ by $V$ and by $W$, where $V \subset \calV, W \subset \calW$ with $\dim_{\R}V=m, \dim_{\mathbb{K}_G}W=n$, so symbolically one may write
\[
(X,m,n) =
\begin{cases}
\Sigma^{-m\R}\Sigma^{-n\tilde{\R}}X \text{\quad if \quad} G=\Z_2,\\
\Sigma^{-m\tilde{\R}}\Sigma^{-n\C}X \text{\quad if \quad} G=\Z_4.
\end{cases}
\]

For $G=\Z_2$, as in \cite{Ma03}, we have {\it canonical} stable equivalences
\begin{align*}
&((V\oplus V)^+ \wedge X, 2\dim_{\R}V + m, n) \simeq (X, m, n),\\
&((W\oplus W)^+ \wedge X, m, 2\dim_{\R}W+n) \simeq (X, m, n).
\end{align*}
The reason why we take twice of $\dim_{\R}V$ and $\dim_{\R}W$ and put a copy of $V$ and $W$ at the first factor is that $GL(N,\R)$ for $N>0$ is not connected and thus there is no canonical choice of trivialization of a given real representation.
When $G=\Z_4$, we actually have a simpler formula for $W$: there is a canonical stable equivalence
\[
(W^+ \wedge X, m, \dim_{\R}W+n) \simeq (X, m, n),
\]
since $\calW$ is a complex representation of $G$.

As well as the non-equivariant case, we can define the notion of local equivalence, which was introduced by Stoffregen~\cite{Sto20}, in our $G$-equivariant setting:

\begin{defi}
\label{defi: stable map and local map}
Let $(X,m,n), (X',m',n')$ be triples as above and let $l \in \Q$.
A {\it $G$-stable map} $(X,m,n) \to (X',m',n')$ of height $l$ is a pointed $G$-map
\[
\Sigma^{V} \Sigma^{W}X
\to \Sigma^{V'} \Sigma^{W'}X'
\]
for some subspaces $V, V' \subset \calV$ and $W, W' \subset \calW$ with $\dim_\R V - \dim_\R V' = m'-m$ and and $\dim_{\mathbb{K}_G}W - \dim_{\mathbb{K}_G}W'=n'-n+l$.
A $G$-stable map $(X,m,n) \to (X',m',n')$ is called a {\it $G$-local map} if $\dim_{\R}V - \dim_{\R}V'=m'-m$ and it induces a $G$-homotopy equivalence on the $H$-fixed-point sets.
We say that $(X,m,n)$ and $(X',m',n')$ are {\it $G$-locally equivalent} if there exist $G$-local maps $(X,m,n) \to (X',m',n')$ and $(X',m',n') \to (X,m,n)$.

The $G$-local equivalence is evidently an equivalence relation, and we call an equivalence class for this relation a {\it $G$-local equivalence class}.
The set of $G$-local equivalence classes is denoted by $\mathcal{LE}_G$. We write an element of $\mathcal{LE}_G$ by $[(X,m,n)]_{\mathrm{loc}}$. 
Evidently the $G$-stable equivalence implies the $G$-local equivalence,
and thus we have a natural surjection $\mathfrak{C}_{G} \to \mathcal{LE}_G$.
\end{defi}

Fix $G=\Z_2$ or $\Z_4$.
For a triple $(X,m,n)$ above, we define
\[
\tilde{H}_G^\ast(X,m,n) := \tilde{H}_G^{\ast+m+k_{G}\cdot n}(X).
\]
Then we may assign each element $\mathcal{X} = [(X, m, n)] \in \mathfrak{C}_G$ to the isomorphism class of equivariant singular cohomology,
    \[
    \tilde{H}_G^\ast(\mathcal{X}) = [\tilde{H}_G^\ast(X,m,n)],
    \]
as a graded $\wt{H}^*_G(S^0)$-module. 
If $G=\Z_2$, we set
\begin{align*}
d(X,m,n) := d(X) - m- n  \in \Q.
\end{align*}
If $G=\Z_4$, we set
\begin{align*}
&\overline{d}(X,m,n) := \overline{d}(X)  - m - 2n \in \Q,\\
&\underline{d}(X,m,n) := \underline{d}(X)  - m- 2n \in \Q.
\end{align*}
Obviously, the invariants $d, \overline{d}, \underline{d}$ descend to invariants
\[
d(\mathcal{X}), \quad 
\overline{d}(\mathcal{X}), \quad 
\underline{d}(\mathcal{X})
\]
for $\mathcal{X} \in \mathfrak{C}_G$.

The invariants $d(\mathcal{X})$, $\overline{d}(\mathcal{X})$ and $\underline{d}(\mathcal{X})$ satisfy the following inequalities: 

\begin{lem}
\label{lem: sourse of Froyshov ineq1}
First we suppose $(G, H)=(\Z_2, \Z_2)$. 
Let $(X,m,n), (X',m',n')$ be triples as above.
If there is a stable local $G$-map of height $l$, then
\[
d(X, m,n )+ l  \leq d(X',m',n').
\]

Next, we suppose $(G, H)= (\Z_4, \Z_2)$.
If there is a stable local $G$-map of height $l$, then we have
\[
\overline{d}(X, m,n )+ 2l \leq \overline{d}(X', m',n' ), \quad
\underline{d}(X, m,n )+ 2l  \leq \underline{d}(X', m',n ').
\]
\end{lem}
\begin{proof}
This follows from \cref{lem: sourse of Froyshov ineq}. 
\end{proof}

As a corollary of \cref{lem: sourse of Froyshov ineq1}, we have the following. 
\begin{cor}\label{lem:loc eq}
The rational numbers $d(\mathcal{X})$ is a $\Z_2$-local equivalence invariant and  $\overline{d}(\mathcal{X})$ and $\underline{d}(\mathcal{X})$ are $\Z_4$-local equivalence invariants.
\end{cor}

\begin{lem}
Let $(G,H)=(\Z_4,\Z_2)$ and let $(X,m,n), (X',m',n')$ be triples as above.
Suppose that there is a $G$-stable map
\[
f : \Sigma^V\Sigma^W X \to \Sigma^{V'}\Sigma^{W'} X'
\]
from $(X,m,n)$ to $(X',m',n')$ of height $l$ as in \cref{defi: stable map and local map} such that $f^{H} : \Sigma^V X^H \to \Sigma^{V'} (X')^H$ is induced from a linear injection whose image is of codimenion-1.
Then we have 
\[
\underline{d}(X,m,n) + 2l
\leq \overline{d}(X',m',n').
\]
\end{lem}

\begin{proof}
This follows from  \cref{Theorem B oiginal}.    
\end{proof}

\subsection{Seiberg--Witten Floer stable homotopy type for non-spin case}
\label{subsection $G$-equivariant Seiberg--Witten Floer homotopy type}

In previous paper, for a spin rational homology 3-sphere with odd involution, we defined Seiberg--Witten Floer stable homotopy type. In this section, we extend our theory to a theory of non-spin case.  

Let $(Y,\frakt)$ be a spin$^c$ rational homology 3-sphere and $\iota$ be a smooth orientation-preserving involution on $Y$.
Suppose that $\iota$ also preserves the given spin$^c$ structure $\frakt$ on $Y$ and is of odd type.
If $\frakt$ comes from a spin structure, we set $(G,H)=(\Z_4, \Z_2)$, and otherwise set $(G,H)=(\Z_2, \Z_2)$.

Fix an $\iota$-invariant metric $g$ on $Y$.
Once we fix $\lambda \ll 0 \ll \mu$,
by Manolescu's construction \cite{Ma03}, we obtain a $G$-equivariant Conley index $I^\mu_\lambda (Y, \frakt, g)$ for a finite-dimensional approximation of the Seiberg--Witten flow.
Recall that, 
the finite-dimensional approximation of
the configuration space is decomposed into 
\[
V^\mu_\lambda(Y,\frakt, g) \oplus W^\mu_\lambda(Y,\frakt, g).
\]
Here each direct summand is isomorphic to
some finite dimensional subspaces of $\calV$ and $\calW$ as $G$-representations.
The $H$-invariant part of $I^\mu_\lambda (Y, \frakt, g)$ is given by $V^\mu_\lambda(Y,\frakt, g)$.

As explained in \cref{involutionI}, the involution $I$ on the configuration space is determined only by $\iota$ up to $\iota$-invariant gauge transformation.
Henceforth, we pick a choice of $I$.
 Let us set
\[
I^\mu_\lambda (Y,\frakt, \iota, g)
= I^\mu_\lambda (Y,\frakt, g)^I,
\quad
V^\mu_\lambda (Y,\frakt, \iota, g)
= V^\mu_\lambda (Y,\frakt, g)^I,
\quad
W^\mu_\lambda (Y,\frakt, \iota, g)
= W^\mu_\lambda (Y,\frakt, g)^I.
\]

As in \cite{Ma03},
set 
\begin{align}
\label{eq: expression of n for Rohlin}
n(Y, \frakt,g) = \ind_{\mathbb{K}_G} D_W + \frac{2}{k_G} \frac{\sigma (W)}{8} \in \frac{1}{8}\Z,    
\end{align}
where $W$ is a compact spin Riemann 4-manifold bounded by $(Y, g)$ and $\ind_\C D_W$ is the complex spin Dirac index on $W$ with spectral boundary condition.

\begin{defi}\label{def: equivariant Seiberg--Witten Floer stable homotopy type}
Given $Y, \frakt, \iota, g$ as above, define an element $SWF_G(Y, \frakt,\iota) \in \mathfrak{C}_{G}$ by 
\begin{align*}
SWF_G(Y, \frakt,\iota)
:= 
[(I^\mu_\lambda (Y, \frakt, \iota, g),
\dim_{\R}V^0_\lambda,
\dim_{\mathbb{K}_G}W^0_\lambda
+n(Y, \frakt,g)/2)].
\end{align*}
We call $SWF_G(Y, \frakt,\iota)$
the {\it $G$-equivariant Seiberg--Witten Floer stable homotopy type} or {\it Seiberg--Witten Floer $G$-spectrum class} for the involution $\iota$.
\end{defi}

As in \cite{KMT}, $SWF_G(Y, \frakt,\iota)$ is an invariant of $(Y, \frakt,\iota)$:

\begin{prop}
\label{lem inv lamda mu}
The spectrum class $SWF_G(Y, \frakt,\iota) \in \mathfrak{C}_{G}$ is an invariant of $(Y, \frakt,\iota)$, independent of $\lambda, \mu$, and $g$.
\end{prop}

\begin{defi}
A $G$-stable map $(X,m,n) \to (X',m',n')$ is called a {\it $G$-local map} if it induces a $G$-homotopy equivalence on the $H$-fixed-point sets.
We say that $(X,m,n)$ and $(X',m',n')$ are {\it $G$-locally equivalent} if there exist $G$-local maps $(X,m,n) \to (X',m',n')$ and $(X',m',n') \to (X,m,n)$.

The $G$-local equivalence is evidently an equivalence relation, and we call an equivalence class for this relation a {\it $G$-local equivalence class}.
The set of $G$-local equivalence classes is denoted by $\mathcal{LE}_G$. We write an element of $\mathcal{LE}_G$ by $[(X,m,n)]_{\mathrm{loc}}$.  
Evidently the $G$-stable equivalence implies the $G$-local equivalence,
and we have a natural surjection $\mathfrak{C}_G \to \mathcal{LE}_G$.
\end{defi}

\subsection{Seiberg--Witten Floer cohomology for involutions}
\label{subsection $G$-equivariant Seiberg--Witten Floer $K$-theory}

As in the last subsection, let $(Y, \frakt)$ be a spin rational homology 3-sphere and $\iota$ be a smooth orientation-preserving involution $\iota$.
Suppose that $\iota$ also preserves the given spin structure $\frakt$ and is of odd type.

\begin{defi}
For $G=\Z_2$ or $\Z_4$,
define the {\it $G$-Seiberg--Witten Floer cohomology for the involution $\iota$} by
\[
SWFH_G(Y, \frakt,\iota) := \tilde{H}^*_G(SWF_G(Y, \frakt,\iota)),
\]
defined as the isomorphism class of an $\tilde{H}_G(S^0)$-module.
We define also the {\it   Fr{\o}yshov-type invariants} by
\begin{align*}
&\delta_R(Y, \frakt,\iota) := \frac{1}{2} d(SWF_{\Z_2}(Y, \frakt,\iota)) \\
&\bar{\delta}_R (Y, \frakt,\iota) := \frac{1}{2}\overline{d}(SWF_{\Z_4}(Y, \frakt,\iota))   \text{ and }\\
&\underline{\delta}_R (Y, \frakt,\iota) :=  
\frac{1}{2}\underline{d}(SWF_{\Z_4}(Y, \frakt,\iota)).
\end{align*}
Note that 
\[
\delta_R(Y, \frakt,\iota), \underline{\delta}_R (Y, \frakt,\iota), \bar{\delta}_R (Y, \frakt,\iota) \in \frac{1}{16}\Z
\]
since we have that $n(Y, \frakt,g) \in \frac{1}{8}\Z$.
\end{defi}

%In the definition of $\kappa(Y, \frakt,\iota)$, we took the half of $k(DSWF_G(Y, \frakt,\iota))$ to cancel the effect of the doubling construction.

\begin{lem}
\label{lem: invariance of k-theory}
The isomorphism class $SWF_G(Y, \frakt,\iota)$ and the rational numbers $\delta_R(Y, \frakt,\iota)$, $\bar{\delta}_R(Y, \frakt,\iota)$ and $\underline{\delta}_R(Y, \frakt,\iota)$ are invariants of $(Y, \frakt,\iota)$.
\end{lem}

\begin{proof}
This is a direct consequence of \cref{lem inv lamda mu}.
\end{proof}

When $Y$ is an integral homology 3-sphere, we often drop the unique spin structure $\frakt$ on $Y$ from our notation.

\begin{ex}
\label{ex Brieskorn sphere with Montesinos}
For a Brieskorn sphere 
\[
Y=\Sigma(p,q,r) = \Set{z_1^p+z_2^q+z_3^r=0},
\]
let $\iota : Y \to Y$ be the involution defined as the complex conjugation.
Namely, define $\iota(z_1,z_2,z_3) = (\bar{z}_1, \bar{z}_2, \bar{z}_3)$.
The fixed-point set $Y^\iota$ is of codimenion-2, and hence $\iota$ is of odd type.
In fact, $Y^\iota$ is known as the Montesinos knot $M(p,q,r)$.

Recall that an odd involution gives rise to a $\Z_4 \times _{\Z_2} \Pin(2)$-symmetry on the Seiberg--Witten equations (\cite{Br98,Ian22}).
For several series of Brieskorn spheres,
Montague~\cite[Subsubsection~8.3.1]{Ian22} computed the $\Z_4 \times _{\Z_2} \Pin(2)$-equivariant Seiberg--Witten Floer spectra and their $I$-invariant parts (which correspond to the $\left<j\mu\right>$-invariant parts in the notation of \cite{Ian22}).
In particular, in our notation, the local equivalence class of $SWF_{\Z_4}(Y,\iota)$ for $Y=\Sigma(2,3,6n\pm1)$ is given by
\begin{align*}
&[SWF_{\Z_4}(\Sigma(2,3,12k-1),\iota)]_{\mathrm{loc}}
= [(\tilde{G},0,0)]_{\mathrm{loc}},\\
&[SWF_{\Z_4}(\Sigma(2,3,12k-5),\iota)]_{\mathrm{loc}}
= [(\tilde{G},0,1/4)]_{\mathrm{loc}},\\
&[SWF_{\Z_4}(\Sigma(2,3,12k+1),\iota)]_{\mathrm{loc}}
= [(S^0,0,0)]_{\mathrm{loc}},\\
&[SWF_{\Z_4}(\Sigma(2,3,12k+5),\iota)]_{\mathrm{loc}}
= [(S^0,0,-1/4)]_{\mathrm{loc}}
\end{align*}
for $k \geq 1$. 
Using this combined with \cref{ex: easiest non-trvial ex},
we can immediately calculate $\delta$, $\underline{\delta}_R$, $\bar{\delta}_R$ for $(\Sigma(2,3,6n\pm1),\iota)$, which is summarized in the following table:

\begin{center}
\begin{tabular}{| c | |  c | c | c | c |}
\hline
Brieskorn sphere equipped with $\iota$ & {$\overline{\delta}_R$} & {$\delta_R$} & {$\underline{\delta}_R$} \\
\hline
\hline
$(\Sigma(2,3,12k-1),\iota)$  & $1/2$ & $1/2$ & $0$  \\
\hline
$(\Sigma(2,3,12k-5),\iota)$  & $0$ & $0$ & $-1/2$  \\
\hline
$(\Sigma(2,3,12k+1),\iota)$  & $0$ & $0$ & $0$  \\
\hline
$(\Sigma(2,3,12k+5),\iota)$  & $1/2$ & $1/2$ & $1/2$ \\
\hline
\end{tabular}
\end{center} 
\end{ex}

\subsection{Cobordisms}
\label{subsection: cobordisms}
Let $(Y_0, \frakt_0)$ and $(Y_1, \frakt_1)$ %\todo{Many connected components version is OK ? (Miyazawa)}
be spin$^c$ closed $3$-manifolds with $b_1(Y_i)=0$. We do not assume that $Y_0$ and $Y_1$ are connected. 
Suppose that we have an involution $\iota_i$ on each of $Y_i$ is an odd type. %which satisfies the condition of \cref{3diminvolutionI}.  %that preserves the orientation and spin structure, and suppose that $\iota_i$ is of odd type.
Let $(W,\fraks)$ be a smooth spin 4-dimensional oriented cobordism with $b_1(W)=0$.
We assume that there is an odd involution $\iota$ on $W$ such that $\iota|_{Y_i}=\iota_i$ for $i=0,1$. 
%From \cref{involutionI}, we have an anti-linear 
Let $\mathbb{S}^{\pm}$ be positive and negative spinor bundles on $W$ and let $\mathbb{S}_i$ be the spinor bundles on $Y_i$ $i=0,1$. 

\begin{comment}
$\iota|_{Y_i}=\iota_i$ for $i=0,1$, and suppose that $\iota$ preserves $\fraks$ and of odd type.
%Let $S$ is the fixed point points of $\iota$.
We may take an $\iota$-invariant Riemannian metric $g$ on $W$ so that $g$ is a cylindrical metric near $\partial W$.
Then the metrics defined by $g_i=g|_{Y_i}$ on $Y_i$ are $\iota_i$-invariant metrics.
Then $\iota$ lifts to some $\tilde{\iota}$, which is a $\Z_4$-lift of $\iota$ to the spinor bundle of $\mathfrak{s}$. 
As well as the case of dimension 3, 
following \cite{Ka17},
we may define the involutions 
\begin{align*}
I : \Omega^\ast(W) \to \Omega^\ast(W),\quad
I : \Gamma(\mathbb{S}^{\pm}) \to \Gamma(\mathbb{S}^{\pm})
\end{align*}
by
\begin{align*}
I(a)=-\iota^{\ast}a,\quad
I(\phi) = \tilde{\iota}( \phi)\cdot j,
\end{align*}
where $\mathbb{S}^{\pm}$ are positive and negative spinor bundles.
\end{comment}
Here we consider the Sobolev norms $L^2_{k}$ for the spaces $\Omega^\ast(W)$ and $\Gamma(\mathbb{S}^{\pm})$ obtained from $\iota$-invariant metrics and $\iota$-invariant connections for a fixed integer $ k \geq 3$.
The relative Bauer--Furuta invariant of $W$ introduced by Manolescu~\cite{Ma03} gives a map between the Seiberg--Witten Floer stable homotopy types of $Y_0$ and $Y_1$.
This is obtained from the Seiberg--Witten map on $W$, which is given as a finite-dimensional approximation of a map
\begin{align}
    \label{eq: SW map}
SW : \Omega_{CC}^1(W) \times \Gamma(\mathbb{S}^+)
\to \Omega^+(W) \times \Gamma(\mathbb{S}^-) \times \hat{V}(-Y_0)^{\mu}_{-\infty} \times \hat{V}(Y_1)^{\mu}_{-\infty}
\end{align}
for large $\mu$. Here
\[
\Omega_{CC}^1(W) 
=\Set{ a \in \Omega^1(W) | d^{\ast} a=0, d^{\ast}\textbf{t}_i a=0, \int_{Y_i} \textbf{t}_i \ast a=0 }
\]
is the space of $1$-forms satisfying the double Coulomb condition. This is introduced by Khandhawit in \cite{Kha15}. 
Here for a general rational homology sphere $Y$ equipped with the spin$^c$ structure, $\hat{V}(Y, \frakt)^\mu_{-\infty}$ is a subspace of
$\hat{V}(Y, \frakt) = \operatorname{Ker}d^* \times \Gamma(\mathbb{S})$
which is defined as the direct sum of eigenspaces whose eigenvalues are less than $\mu$.
The $\Omega^+(W) \times \Gamma(\mathbb{S}^-)$-factor of the map $SW$ is given as the Seiberg--Witten equations,
and the $\hat{V}(-Y_0)^{\mu}_{-\infty} \times \hat{V}(Y_1)^{\mu}_{-\infty}$-factor is given, roughly, as the restriction of 4-dimensional configurations to 3-dimensional ones.
Taking the $I$-invariant part of \eqref{eq: SW map},
we obtain a $G$-equivariant map and a finite-dimensional approximation of this gives us a $G$-equivariant map of the form
\begin{align}
\label{eq: cob map}
f : \Sigma^{m_0\tilde{\R}}\Sigma^{n_0 \mathbb{K}_G}I_{-\mu}^{-\lambda} (Y_0) \to \Sigma^{m_1\tilde{\R}}\Sigma^{n_1 \mathbb{K}_G}I^\mu_\lambda (Y_1),
%f : \Sigma^{m_0\tilde{\R}}\Sigma^{n_0^+\C_+}\Sigma^{n_0^-\C_+}I_{-\mu}^{-\lambda} (Y_0) \to \Sigma^{m_1\tilde{\R}}\Sigma^{n_1^+\C_+}\Sigma^{n_1^-\C_+}I^\mu_\lambda (Y_1),
\end{align}
where $I^\mu_\lambda (Y_i) = I^\mu_\lambda (Y_i, \frakt_i, \iota_i, g_i)$, and $m_i, n_i^\pm \geq 0$ and $-\lambda, \mu$ are sufficiently large. The representation spaces $\tilde{\R}$ and $\mathbb{K}_G$ are the representations of $G$ which we introduced in \cref{subsection Doubling}. 
%\cref{representations}. 

Let us denote by $V_i(\tilde{\R})_\lambda^\mu$ the vector space $V(\tilde{\R})_\lambda^\mu$, a finite-dimensional approximation of the $I$-invariant part, for $Y_i$, and let us use similar notations also for other representations.

\begin{lem}
\label{lem: differences of dim}
We have
\begin{align}
\label{eq:m difference}
m_0-m_1=&\dim_\R (V_1(\tilde{\R})^0_{\lambda})-\dim_\R(V_0(\tilde{\R})^0_{-\mu}) - b^+(W) + b^+_\iota(W), \\
\begin{split}
n_0-n_1=&\dim_{\mathbb{K}_G}(V_1(\mathbb{K}_G)^0_{\lambda}) %+\dim_\C(V_1(\C_-)^0_{\lambda})
-\dim_{\mathbb{K}_G}(V_0(\mathbb{K}_G)^0_{-\mu})%-\dim_\C(V_0(\C_-)^0_{-\mu})
\\
&- \frac{1}{k_G}\left( \frac{c_1(\fraks)^2 - \sigma(W)}{8}+ n(Y_1, \mathfrak{t}_1, g_1)-n(Y_0, \mathfrak{t}_0, g_0) \right).\label{eq:n difference}
\end{split}
\end{align}
%\begin{align}
%\label{eq:m difference may be wrong}
%m_0-m_1=&\dim_\R (V_1(\tilde{\R})^0_{\lambda})-\dim_\R(V_0(\tilde{\R})^0_{\lambda}) - b^+(W) + b^+_\iota(W), \\
%\begin{split}
%n_0-n_1=&\dim_\C(V_1(\C_+)^0_{\lambda}) +\dim_\C(V_1(\C_-)^0_{\lambda})\\
%&-\dim_\C(V_0(\C_+)^0_{\lambda})-\dim_\C(V_0(\C_-)^0_{\lambda})\\
%&- \frac{\sigma(W)}{16}+ \frac{n(Y_1, \mathfrak{t}_1, g_1)}{2}-\frac{n(Y_0, \mathfrak{t}_0, g_0)}{2}.\label{eq:n difference}
%\end{split}
%\end{align}
\end{lem}
\begin{proof}
    From  \cite[Theorem 1.]{Kha15}, we can calculate the index of the linearization of the map \eqref{eq: SW map}. Taking the $(-\iota^\ast, I)$-invariant part of the finite dimensional approximation of $f$, we have that $m_0-m_1$ and $n_0-n_1$ coincides with the index of $-\iota^\ast$ and $I$ fixed part of the form part and spinor part respectively. Since $I$ is an anti-linear involution, $n_0-n_1$ is half of the Dirac index. 
\end{proof}

\subsection{Proof of Fr\o yshov type inequalities} 
\label{subsection Proof of main theo} 

Let $Y$ be an oriented closed rational homology 3-sphere.
Let $(\frakt,\iota)$ be a real spin$^c$ structure on $Y$, i.e. $\frakt$ is a spin$^c$ structure on $Y$ and $\iota : Y \to Y$ is an orientaion-preserving smooth involution on $Y$ such that $\iota^\ast\frakt \cong \bar{\frakt}$.
As in dimension 4, we shall define the notion that $\iota$ is of {\it odd type}, which is satisfied if the fixed-point set $Y^\iota$ is non-empty and of codimension-2.
For this triple $(Y,\frakt,\iota)$,
we have the Fr{\o}yshov-type invariant
\[
\delta_R(Y,\frakt,\iota) \in \frac{1}{16}\Z.
\]

First, we state and prove the main inequality of this paper in the most general form, which relates $\delta_R(Y,\frakt,\iota)$ with topological quantities:

\begin{theo}
\label{most general main theo}
Let $(Y_0, \frakt_0)$, $(Y_1, \frakt_1)$ be spin$^c$ rational homology 3-spheres.
Let $\iota_0, \iota_1$ be smooth involutions on $Y_0, Y_1$.
Suppose that $\iota_0, \iota_1$ preserve the given orientations and spin$^c$ structures $\frakt_0, \frakt_1$ on $Y_0, Y_1$ respectively, and suppose that $\iota_0, \iota_1$ are of odd type.
Let $(W,\fraks)$ be a smooth compact oriented spin$^c$ cobordism with $b_1(W)=0$ from $(Y_0, \frakt_0)$ to $(Y_1, \frakt_1)$.
Suppose that there exists a smooth involution $\iota$ on $W$ such that $\iota$ preserves the given orientation and spin$^c$ structure $\fraks$ on $W$, $b^+(W)- b^+_\iota(W)=0$ and that the restriction of $\iota$ to the boundary is given by $\iota_0, \iota_1$.
Then we have
\begin{align}
\delta_R(Y_0,\frakt_0,\iota_0)+
\frac{c_{1}(\fraks)^{2} - \sigma(W)}{16} \leq \delta_R(Y_1,\frakt_1,\iota_1).
\end{align}
\end{theo}

\begin{proof}
The relative Bauer--Furuta invariant for $(W, \fraks, \iota)$ gives a map
\[
f : \Sigma^{V}\Sigma^{W}SWF (Y_0, \frakt_0, \iota_0)\to \Sigma^{V'}\Sigma^{W'}SWF (Y_1, \frakt_1, \iota_1)
\]
and we have $\dim_{\R}(V)-\dim_{\R}(V')=b^+(W)-b^+_{\iota}(W)$ and $\dim_{\R}(W)-\dim_{\R}(W')=(c_{1}(\fraks)^{2} - \sigma(W))/8$. 
The claim of \lcnamecref{most general main theo} follows from
\cref{lem: sourse of Froyshov ineq} applied to this $f$.
\end{proof}

If we focus on the case that $\frakt$ comes from a spin structure, for which one has $\frakt \cong \bar{\frakt}$, we may define invariants analogous to invariants introduced by Stoffregen~\cite{Sto171}   $\bar{\delta}_R(Y,\frakt),\  \underline{\delta}_R(Y,\frakt)$ corresponding to $\bar{d}(Y,\frakt),\ \underline{d}(Y,\frakt)$ in involutive Heegaard Floer homology \cite{HM17}.
For a spin rational homology 3-sphere $(Y,\frakt)$ and an orientaion-preserving smooth involution $\iota : Y \to Y$ such that $\iota^\ast\frakt \cong \frakt$, when $\iota$ is of odd type, we have two invariants
\[
\bar{\delta}_R(Y,\frakt,\iota),\  \underline{\delta}_R(Y,\frakt,\iota) \in \frac{1}{16}\Z,
\]
which satisfy inequalities
\[
\underline{\delta}_R(Y,\frakt,\iota)
\leq \delta_R(Y,\frakt,\iota)
\leq \bar{\delta}_R(Y,\frakt,\iota).
\]
On these invariants, we have the following:

\begin{theo}
\label{most general main theo2}
Let $(Y_0, \frakt_0)$, $(Y_1, \frakt_1)$ be spin rational homology 3-spheres.
Let $\iota_0, \iota_1$ be smooth involutions on $Y_0, Y_1$.
Suppose that $\iota_0, \iota_1$ preserve the given orientations and spin structures $\frakt_0, \frakt_1$ on $Y_0, Y_1$ respectively, and suppose that $\iota_0, \iota_1$ are of odd type.
Let $(W,\fraks)$ be a smooth compact oriented spin cobordism with $b_1(W)=0$ from $(Y_0, \frakt_0)$ to $(Y_1, \frakt_1)$.
Suppose that there exists a smooth involution $\iota$ on $W$ such that $\iota$ preserves the given orientation and spin structure $\fraks$ on $W$, $b^+(W)- b^+_\iota(W)=0$ and that the restriction of $\iota$ to the boundary is given by $\iota_0, \iota_1$.
Then we have
\begin{align}
&\bar{\delta}_R(Y_0,\frakt_0,\iota_0) - \frac{\sigma(W)}{16} \leq \bar{\delta}_R(Y_1,\frakt_1,\iota_1),\\
&\underline{\delta}_R(Y_0,\frakt_0,\iota_0) - \frac{\sigma(W)}{16} \leq \underline{\delta}_R(Y_1,\frakt_1,\iota_1).
\end{align}

Moreover, if $b^+(W)- b^+_\iota(W)=1$, then we have 
\begin{align}\label{real Theorem B}
    \underline{\delta}_R(Y_0,\frakt_0,\iota_0) - \frac{\sigma(W)}{16} \leq \bar{\delta}_R(Y_1,\frakt_1,\iota_1). 
\end{align}
\end{theo}
\begin{proof}
When $\mathfrak{s}$ is spin and $b^+(W)-b^+_\iota(W)=0$, the Bauer--Furuta invariant 
for $(W, \fraks, \iota)$,
\[
f : \Sigma^{V}\Sigma^{W}SWF (Y_0, \frakt_0, \iota_0)\to \Sigma^{V'}\Sigma^{W'}SWF (Y_1, \frakt_1, \iota_1),
\]
is a $\Z_4$-equivariant local map of height $-\sigma(W)/8$ since 
$\dim_{\R}(V)-\dim_{\R}(V')=b^+(W)-b^+_{\iota}(W)$ and $\dim_{\R}(W)-\dim_{\R}(W')= - \sigma(W)/8$. 
So, the first two inequalities follow from
\cref{lem: sourse of Froyshov ineq1} applied to this $f$. The third inequality follows from \cref{Theorem B oiginal} applied to $\Z_4$-equivariant relative Bauer--Furuta invariants of $(W, \frak{s})$.
\end{proof}

 The inequality \eqref{real Theorem B} can be seen as a real and 4-manifold with boundary version of Donaldson's Theorem B \cite{Do83}.

\subsection{Connected sum formula} 
In this section, we prove the following connected sum formula of $[SWF(Y, \frakt, \iota)]_{\mathrm{loc}}$ and $\delta$. We first write a statement of connected sum for $[SWF(Y, \frakt, \iota)]_{\mathrm{loc}}$.

First, we recall the operation of  {\it equivariant connected sum}. For more details, see \cite[Definition 3.43]{KMT}. 
Let $(Y_0, \mathfrak{t}_0, y_0)$ and $ (Y_1,\mathfrak{t}_1, y_1)$ be spin$^c$ rational homology $3$-spheres with base points. Let $\iota_0$ and $\iota_1$ are involutions on $Y_0$ and $Y_1$ respectively. Suppose that the fixed-point set of $\iota_i$ are codimension-2 and $\iota_i(y_i)=y_i$, and that  $\iota_i$ preserves the spin structure $\mathfrak{t}_i$ for $i=0, 1$.
We give an orientation $o(\iota_i)$ of the set of fixed points of the involution $\iota_i$. 
For these data, the equivariant connected sum
\[
(Y_0, \mathfrak{t}_0, y_0, \iota_0, o(\iota_0)) \# (Y_1, \mathfrak{t}_1, y_1, \iota_1, o(\iota_1))
\]
is defined. 
We sometimes drop $y_0, y_1$ and $o(\iota_1), o(\iota_2) $ from our notation.

The standard 3-handle cobordism between the connected sum and disjoint union of two manifolds can be considered equivariantly.
Namely, we have an oriented homology cobordism 
\[
W_{01}: (Y_0, \mathfrak{t}_0, \iota_0, o(\iota_0)) \# (Y_1, \mathfrak{t}_1, \iota_1, o(\iota_1))\to (Y_0, \mathfrak{t}_0, \iota_0, o(\iota_0)) \sqcup (Y_1, \mathfrak{t}_1, \iota_1, o(\iota_1))
\]
that is equipped with an involution $\iota$ whose restriction to the boundary is given by $\iota_0, \iota_1$.

\begin{lem}
The cobordism $W_{01}$ admits a spin$^c$ structure $\fraks$ whose restriction to the boundary is given by $\frakt_0, \frakt_1$ such that $\iota^* \fraks \cong  \overline{\fraks}$.
 \end{lem}
 \begin{proof}
 From homological computations, one can verify $H^2(W_{01}, \partial W_{01}; \Z)=0$. Thus, we know an extension of spin$^c$ structure of $\frakt_0\cup \frakt_1 \cup \frakt_0\# \frakt_1$ exists and is unique up to isomorphism. 
 \end{proof}

The connected sum theorem can be formulated as follows. 
\begin{theo}
\label{conn sum of local eq}
Let $(Y_0, \mathfrak{t}_0, \iota_0)$ and $ (Y_1,\mathfrak{t}_1, \iota_1)$ are spin$^c$ $\Z_2$-homology $3$-spheres with odd involutions $\iota_0$ and $\iota_1$. 
Suppose that the fixed-point set of $\iota_i$ is non-empty and connected for each $i=0,1$.
Then, we have 
\[
[SWF_G (Y_0, \mathfrak{t}_0, \iota_0) \wedge SWF_G (Y_1, \mathfrak{t}_1, \iota_1) ]_{\rm{loc}} = [SWF_G(Y^\#, \mathfrak{t}^\#, \iota^\#)]_{\rm{loc}}, 
\]
where $(Y^\#, \mathfrak{t}^\#, \iota^\#) := (Y_0, \mathfrak{t}_0,\iota_0) \# (Y_1, \mathfrak{t}_1, \iota_1)$. 
\end{theo}

\begin{proof}
This follows from the same argument in the proof of \cite[Theorem 3.48]{KMT} although the doubled Floer homotopy types are used in \cite{KMT}. 
The local maps of both directions are given as relative Bauer--Furuta invariants for spin$^c$ cobordism $\pm W_{01} $ with involution.  
\end{proof}

\begin{theo}\label{conn sum of delta} Let $(Y_0, \mathfrak{t}_0, \iota_0)$ and $ (Y_1,\mathfrak{t}_1, \iota_1)$ are spin$^c$ $\Z_2$-homology $3$-spheres with odd involutions $\iota_0$ and $\iota_1$.
\[
\delta_R(Y, \iota, \mathfrak{t}) + \delta_R(Y', \iota', \mathfrak{t}') = \delta_R(Y\# Y', \iota \# \iota' , \mathfrak{t} \#\mathfrak{t}' ) .
\]
\end{theo}

In order to prove \cref{conn sum of delta}, we first prove the following duality result: 
\begin{lem}
\label{prop: from duality}
Let $(Y,\frakt,\iota)$ be an oriented spin rational homology sphere with odd involution.
Then we have
\[
\delta_R(Y,\frakt,\iota) =  -\delta_R(Y,\frakt,\iota)  . 
\]
Moreover, suppose $\frakt$ is spin, we have 
\[
\underline{\delta}_R(Y,\frakt,\iota) =- \bar{\delta}_R(-Y,\frakt,\iota). 
\]
\end{lem}

\begin{proof}[Proof of \cref{prop: from duality}]

We first claim that $I^\lambda_{-\lambda} (Y, \iota, g)$ and $I^{\lambda}_{-\lambda} (-Y, \iota, g) $ are equivariantly $V^\lambda_{-\lambda}(Y, \iota, g) \oplus W^\lambda_{-\lambda}(Y, \iota, g)$-dual. This can be obtained as the restriction of the $V$-dual map \cite[Proposition 6.23]{Ian22}. 
%As it is pointed out in \cite[Proposition 3.8]{Ma16}, for a Riemann 3-manifold $Y$,  finite dimensional approximations of Coulomb slice satisfy $V^\tau_\nu (Y) = V^{-\nu}_{-\tau} (-Y)$. In particular, we have $V^\tau_\nu (Y)^I = V^{-\nu}_{-\tau} (-Y)^I$. 
Then the inequalities of the conclusion directly follow from \cref{duality for d}. 
\end{proof}

\begin{proof}[Proof of \cref{conn sum of delta}]
From \cref{conn sum for d} and \cref{conn sum of local eq}, the following inequality has already been proven: 
\[
\delta_R(Y, \iota, \mathfrak{t}) + \delta_R(Y', \iota', \mathfrak{t}') \leq  \delta_R(Y\# Y', \iota \# \iota' , \mathfrak{t} \#\mathfrak{t}' ) .
\]

The remaining part is to see the opposite inequality. To see this, using the above inequality to orientation reversals, we obtain
\begin{align*}
\delta_R(-Y, \iota, \mathfrak{t}) + \delta_R(-Y', \iota', \mathfrak{t}') &\leq  \delta_R(-(Y\# Y'), \iota \# \iota' , \mathfrak{t} \#\mathfrak{t}' ) . 
\end{align*}
Then, we use \cref{prop: from duality} and obtain 
\begin{align*}
\delta_R(Y, \iota, \mathfrak{t})+ \delta_R(Y', \iota', \mathfrak{t}') &\geq   \delta_R(Y\# Y', \iota \# \iota' , \mathfrak{t} \#\mathfrak{t}' ) . 
\end{align*}
This completes the proof. 
\end{proof} 

\begin{comment}
\begin{prop}
Let $(Y,\frakt,\iota)$ be an oriented spin rational homology sphere with odd involution.
Then we have
\[
\underline{\delta}_R(Y,\frakt,\iota) =- \bar{\delta}_R(-Y,\frakt,\iota). 
\]
\end{prop}

\begin{proof}
\end{proof}

\end{comment}

\subsection{SWF-spherical}
We introduce the easiest case to calculate our invariants $d$, $\underline{d}$ and $\overline{d}$. 
\begin{defi}\label{SWF-spherical}
We say that a triple $(Y, \mathfrak{t}, \iota)$ of a rational homology 3-sphere, spin$^c$ structure, and an odd involution is {\it SWF-spherical} if
\[
[SWF_G (Y, \mathfrak{t}, \iota )]_{\mathrm{loc}} = [(S^0, m,n)]_{\mathrm{loc}}
\]
for some $m\in \Z$ and $n \in \Q$.
\end{defi}
The following computations are obvious from the definitions of invariants $\delta_R$, $\underline{\delta}_R$ and $\bar{\delta}_R$.
\begin{prop}\label{spherical}
If a triple $(Y, \mathfrak{t}, \iota)$ is SWF-spherical, then 
\[
\delta_R(Y, \iota, \frakt) = \underline{\delta}_R(Y, \iota, \frakt) = \bar{\delta}_R(Y, \iota, \frakt)=   -\frac{m}{2 } - \frac{n}{2}
\]
for $m,n$ satisfying $[SWF_G (Y, \mathfrak{t}, \iota )]_{\mathrm{loc}} = [(S^0, m,n)]$.
\end{prop}

Now, we prove \cref{computation} using \cref{spherical}. 
\begin{proof}[Proof of \cref{computation}]
From \cite{KMT}, for lens spaces, we have the following computation: 
Let $p$, $q$ be coprime natural numbers.
Regard the lens space $Y=\Sigma(K(p,q)) =
L(p, q)$
, and equip $Y$ with the standard metric $g$, which has
positive scalar curvature. The complex conjugation on $Y$ defines an involution $\iota$ on $Y$ that preserves $g$. The fixed-point set of $\iota$ is non-empty and of codimension-2, which is called two bridge knot/link. 
Through the same discussion given in \cite[Example 3.56]{KMT}, we have 
\[
SWF_G(L(p, q), t,\iota) = [(S^0, 0, n(L(p, q),\frakt, g)/2].
\]
This combined with \cref{spherical} completes the computations for $\delta$, $\bar{\delta}_R$ and $\underline{\delta}_R$. 
For torus knots, the proof is similar and we use the argument of the proof of \cite[Theorem 3.58]{KMT} instead of \cite[Example 3.56]{KMT}. 
\end{proof}

\section{Floer homotopy type of links and Fr\o yshov type invariants}\label{section Floer homotopy type of knots and Fro yshov type invariants}

In \cite{KMT}, we defined a Floer homotopy type of a knot
\[
K \mapsto DSWF(K)
\]
which we called the doubled Floer homotopy type for a knot $K$. For some reason related to $K$-theory, we made the doubling operation. In this paper, we give a Floer homotopy type $SWF(K)$ without taking double. 
Moreover, we generalize the invariant $SWF(K)$ to that of oriented links $L$ in $S^3$ with non-zero determinant: 
\[
L \mapsto SWF (L). 
\]
Using them, for links with non-zero determinant, we introduce Fr\o yshov type invariants
\[
L \mapsto \delta_R (L), \underline{\delta}_R (L), \bar{\delta}_R (L)  \in \frac{1}{16}\Z. 
\]

\subsection{Floer homotopy type of links}

Let $L$ be a link in $S^3$ with non-zero determiant. 
Then, we can associate the double branched cover
 \[
 \Sigma (L) \to S^3 
 \]
 uniquely and a covering involution $\iota: \Sigma (L)  \to  \Sigma (L) $. Since our theory uses real spin$^c$ structures, we need to confirm compatibility between the covering involutions and spin$(^c)$ structures on the double branched covering spaces. 
 The following correspondence between orientations of $L$ and spin structures on $\Sigma(L)$ is proven by Turaev \cite{Tu85}: 

 %{\color{red} Check whether this needs non-zero determiant}
\begin{lem}\label{Traev}
There is a canonical one-to-one correspondence between the set of isomorphism classes of spin structures on $\Sigma (L)$ and $\{ \operatorname{ orientations\ on } L \} / \{\pm 1\}$.
\end{lem}

We need to clarify the correspondence in \cref{Traev} to see if the spin structures are preserved by the involution. 

For a given link $L$, we take an open tubular neighborhood $N_L$ of $\tilde{L}$ in $\Sigma(L)$, where  $\tilde{L}$ is the lift of $L$ in $\Sigma(L)$.
Let $\fraks$ be the unique spin structure on $\Sigma(L)\setminus N_L$ coming from the pull-back of the spin structure on $S^3 \setminus L$, which is the restriction of the spin structure on $S^3$.
We pick an orientation $o$ of $L$.
Corresponding to $ o$, one can associate an element $h (0)$ of $H^1(\Sigma(L)\setminus N_L; \Z_2 )$ as follows: 
\[
h(o) (l) := \frac{1}{2}\operatorname{lk}(p(l),L) : H_1(\Sigma(L)\setminus N_L; \Z_2 ) \to \Z_2,
\] where $p$ is the covering projection of $\Sigma(L)$ and $\operatorname{lk}$  is the  linking number with respect to the orientation $o$.
Thus, we have a spin structure $\fraks + h(o)$ on $\Sigma(L) \setminus N_L$, which has the unique extension to the spin structure on $\Sigma(L)$ denote by $\fraks(o)$. 
Then, the map 
\begin{align}\label{ori and spin structure}
\{ \text{orientations on } L\}/ \{\pm 1\} \to  \operatorname{Spin}(\Sigma( L))\ ; o \mapsto \fraks(o)
\end{align}
gives a bijection, where $\operatorname{Spin}(\Sigma( L)) $ denotes the set of ismorphism classes of spin structures on $\Sigma( L)$. 
\begin{comment}
Then, consider the sequence 
\[
 0 \cong H^1(N_L, \partial N_L ; \Z_2)  \to H^1 (\Sigma(L); \Z_2) \to  H^1(\Sigma(L) \setminus N_L; \Z_2)\to 
\]
\[
H^2( N_L, \partial N_L; \Z_2)\cong \Z_2^{n_L}\to  H^2 (\Sigma(L); \Z_2) \to H^2(\Sigma (L) \setminus N_L; \Z_2)  \to 0.
\]
\[
0 \to  H^1(N_L, \partial N_L ; \Z_2) \to H^1(N_L; \Z_2) \to H^1(\partial N_L; \Z_2) \to H^2(N_L ,\partial N_L; \Z_2)\to 0
\]

\end{comment}
From the constructions of the map \eqref{ori and spin structure}, we can ensure that the spin structure $\frak{s} + h(o)$ is preserved by the covering involution $\iota$. Thus we have: 
\begin{lem}
Any spin structure on $\Sigma(L)$ is preserved by the involution $\iota$.
\end{lem}

For a given orientation of $L$, we denote by $\mathfrak{t}_L$ the spin structure on $\Sigma (L)$ corresponding to the orientation. 
Note that $\mathfrak{t}_L\cong  \mathfrak{t}_{-L} $. 

\begin{defi}
For a given oriented link $L$ with non-zero determinant, we define 
\[
SWF(L) := SWF (\Sigma (L), \iota, \mathfrak{t}_L). 
\]

\end{defi}

Obviously, the $\Z_4$-equivariant stable homotopy type $SWF(L) $ is invariant under isotopy for links.

\subsection{Fr\o yshov type invariants for links}

Now, we define the three invariants: 
\begin{defi}\label{def for link}
For a given oriented link $L$, 
we define 
\[
\delta_R (L) := \delta_R(\Sigma (L), \iota, \mathfrak{t}_L), \ \ \   \bar{\delta}_R (L) := \overline\delta_R(\Sigma (L), \iota, \mathfrak{t}_L),\ \ \  \underline{\delta}_R (L) := \underline\delta_R(\Sigma (L), \iota, \mathfrak{t}_L), 
\]
where $\mathfrak{t}_L$ is the spin structure corresponding to the given orientation of $L$. For the correspondence between orientations of $L$ and spin structures on $\Sigma(L)$, see \eqref{ori and spin structure}.
Similarly, we define 
\[
\kappa (L) := \kappa(\Sigma (L), \iota, \mathfrak{t}_L). 
\]
Here $\kappa(-)$ is the invariant defined in \cite[Theorem 1.1]{KMT}.
\end{defi}
Before proving \cref{theo: main knot}, we need to calculate several fundamental quantities of double branched covering spaces.

Let $(X, S)$ be a connected (possibly non-orientable) link cobordism from $(S^3, L)$ to $(S^3, L')$ with $H_1(X ; \Z_2) =0$. Suppose the determinants $L$ and $L' $ are non-zero. 
As in the closed case given in \cite[Corollary 2.10]{Na00}, we see that the surface $S$ in $X$ determines a unique double branched covering space if and only if $0 \equiv [S]\in H_2(X, \partial X; \Z_2)$.
The followings are fundamental computations regarding the homologies of double-branched covering spaces. 

\begin{lem}
\label{branched homology}
The following equalities hold:
\begin{align}
&\sigma (\Sigma (S))= 2 \sigma (X) - \frac{1}{2} S \circ S - \sigma (L)+ \sigma (L'),  \label{eq: sign lem}\\ 
&b_2 (\Sigma (S )) = 2b_2(X) + b_1(S) , \label{eq: b2}  \\ 
& b^+ (\Sigma (S )) = 2b^+(X) + \frac{1}{2}b_1(S) -\frac{1}{4} S \circ S  -  \frac{1}{2} \sigma (L) + \frac{1}{2} \sigma (L'), \label{eq: bplus lem}\\ 
& b_1(\Sigma(S) ) = b_3(\Sigma(S) )  =0. \label{eq: b1 lem}
 \end{align}
 Moreover, if $L$ and $L'$ are knots, there is no 2-torsion in the cohomology $H^2(\Sigma(S); \Z)$.
\end{lem}
\begin{proof}
The equality \eqref{eq: b1 lem} follows from the exact sequence given in \cite{RS95}.
The Mayer--Vietoris argument given in \cite[Lemma 4.2]{KMT} enables us to confirm \eqref{eq: sign lem} by the following discussion. It is confirmed in \cite[Lemma 1.1]{KT76} that for a Seifert surface $S_L$ in $S^3$ bounded by $L$ giving $\sigma(L)$, the double branched covering $N_L$ for its push into $D^4$ satisfies 
\begin{align}\label{signature equality}
\sigma (L) = \sigma (N_L ).
\end{align}
From $G$-signature theorem for closed 4-manifolds and \eqref{signature equality}, we obtain \eqref{eq: sign lem}. Next, we see \eqref{eq: b1 lem}. Since $\Sigma(S) \to X$ is a degree $2$-branched cover, we have $\chi(\Sigma(S)) = 2 \chi(X) -\chi(S)$. From \eqref{eq: b1 lem}, we have $b_2(\Sigma(S)) =  2b_2(X) +  b_1(S)$. We note \eqref{eq: bplus lem} follows from \eqref{eq: b2} and \eqref{eq: sign lem}. 
We finally show there is no 2-torsion in $H^*(\Sigma(S); \Z)$ when $L$ and $L'$ are knots.
From \cite{RS95}, the sequence 
\[
H_2(X; \Z_2) \to H_1(X, S; \Z_2)\to H_1(\Sigma(S); \Z_2) \to 0. 
\]
is exact. The pair exact sequence for $(X, S)$ combined with connectivity of $S$ implies $H_1(X, S; \Z_2)=0$. Thus we have $H_1(\Sigma(S); \Z_2)=0$, and there is no 2-torsion in $H_1(\Sigma(S);\Z)=0$. This implies there is no 2-torsion on $H^2(\Sigma(S))$. This  completes the proof. 
\end{proof}

We now prove \cref{theo: main knot}(iii). 
\begin{proof}[Proof of \cref{theo: main knot}(iii)]
This follows from \cref{most general main theo} combined with the computation of homological invariants in \cref{branched homology}.  
\end{proof}

We also have similar inequalities for $\bar{\delta}_R$ and $\underline{\delta}_R$.
Note that \cref{Theorem B for links} follows from \eqref{Theorem B for link ineq}. 

\begin{theo}\label{theo: main ineq d bar}
Let $L$ and $L'$ be links in $S^3$ with non-zero determinants, $X$ be an oriented smooth compact connected cobordism from $S^3$ to $S^3$, and $S$ be a compact connected properly and smoothly (possibly non-orientable) embedded cobordism in $X$ from $L$ to $L'$ such that the homology class $[S] \equiv 0 \operatorname{mod} 2$. Suppose $H_1(X; \Z)=0$, ${PD}[S] \equiv w_2(X) \in H^2(X; \Z_2)$, \[
    b^+(X) + g(S) -\frac{1}{4} S \circ S  -  \frac{1}{2} \sigma (L) + \frac{1}{2} \sigma (L')=0
    \]
    and there is a spin structure $\fraks$ on $\Sigma (S)$ such that $\iota^* \fraks \cong \fraks$ and whose restrictions are compatible with orientations of $L$ and $L'$.
    
    Then, we have
    \begin{align} \label{ineq2}
     \bar{\delta}_R(L)+
\frac{1}{16}\left( - 2 \sigma (X) + \frac{1}{2} S \circ S + \sigma (L)- \sigma (L')\right) \leq \bar{\delta}_R(L') \\
\underline{\delta}_R(L)+
\frac{1}{16}\left( - 2 \sigma (X) + \frac{1}{2} S \circ S + \sigma (L)- \sigma (L')\right) \leq \underline{\delta}_R(L')
    \end{align}
    where $S \circ S$ means the self intersection number of $S$ in $X$.

    Moreover, when 
    \[
    b^+(X) + g(S) -\frac{1}{4} S \circ S  -  \frac{1}{2} \sigma (L) + \frac{1}{2} \sigma (L')=1, 
    \]
    we have 
    \begin{align}\label{Theorem B for link ineq}
        \underline{\delta}_R(L)+
\frac{1}{16}\left( - 2 \sigma (X) + \frac{1}{2} S \circ S + \sigma (L)- \sigma (L')\right) \leq \bar{\delta}_R(L')
    \end{align}
\end{theo}

\begin{proof}[Proof of \cref{K general ineq for link}]
These inequalities follow from \cref{most general main theo} and \cref{most general main theo2} combined with \cref{branched homology}.  
\end{proof}

As we discussed in the introduction, our invariant will be $\chi$-concordance invariant. 
Thus, we introduce the notion of {\it $\chi$-concordant}. 
First, we introduce a notion of $\chi$-slice: 
\begin{defi}
A link $L$ in $S^3$
is {\it $\chi$-slice}
if $L$ bounds a smoothly properly embedded
surface $F$ in $D^4$ without closed components, and with $\chi (F)=1$. If $L$ is oriented we
require $F$ to be compatibly oriented for orientable components of $F$.

\end{defi}
Note that $\chi$-sliceness gives a generalization of sliceness of knots. In order to take connected sum, we also need a base point. A {\it marked link} is a link in $S^3$ equipped with a marked component. 
\begin{defi}For given oriented marked links $L_0$ and $L_1$, we say $L_0$ and $L_1$ are {\it $\chi$-concordant} if $-L^*_0\#L_1$ bounds a smoothly properly
embedded surface $F$ in $D^4$
such that
\begin{itemize}
    \item[(i)]  $F$ is a disjoint union of one disk together with annuli and Mobius bands;
    \item[(ii)] the boundary of the disk component of $F$ is the marked component of $-L^*_0 \# L_1$;
    \item[(iii)] we require orientable components of $F$ to be oriented compatible with the orientation of $-L^*_0 \#L_1$ . 
\end{itemize}
Here the conncted sum $-L^*_0\#L_1$ is taken along marked points. 
\end{defi}
If $L_0$ and $L_1$ are $\chi $-concordant, then $-L^*_0 \# L_1$ is $\chi$-slice. On the other hand, the converse is not true. 
 In \cite{DO12}, it is proven that the set $\tilde{\mathcal{L}}$ of all $\chi$-concordant classes of  oriented mariked links forms an abelian group with respect to the connected sum along marked components. The group $\tilde{\mathcal{L}}$ is called the {\it link concordance group}. 
In this paper, we focus on the subgroup $\tilde{\mathcal{F}}$ of $\tilde{\mathcal{L}}$ generated by oriented mariked links whose determinants are {\it non-zero}. 

We first provide a homotopy theoretic homomorphism: 

\begin{prop}
The map $[L] \mapsto [SWF(L)]_{\mathrm{loc}}$ gives a homomorphism from the link concordance group $\wt{\mathcal{F}}$ to $\mathcal{LE}_G$, where $G=\Z_4$.  
\end{prop}

\begin{proof}
As $\Z_2$-equivariant spin manifolds, we have 
\[
(\Sigma (L \# L'), \iota_{L\#L'}, \frakt_{L\# L'} )  \cong (\Sigma (L ), \iota_{L}, \frakt_{L} ) \# (\Sigma ( L'), \iota_{L'},\frakt_{L'}  )., 
\]
where $\iota_{L}$ denotes the covering involution of double branched covering space. Also, taking local equivalence class $[-]_{\mathrm{loc}}$ is a homomorphism from \cref{conn sum of local eq}. This completes the proof. 
\end{proof}
We next prove a relation between our invariants and $\chi$-sliceness. 
\begin{lem}\label{chi-slice zero}
If $L$ is $\chi$-slice, then $\underline{\delta }_R(L)$, $\bar{\delta}_R (L)$ and $\kappa(L)$ are zero. 
Moreover, the quantity $\delta_R (L)$ is a $\chi$-concordace invariant. 
\end{lem}

\begin{proof}
Suppose $L$ is $\chi$-slice, i.e. $-L^*_0 \# L_1$ bounds a smoothly properly embedded surface $S$ in $D^4$ without closed components, and with $\chi (S)=1$. Then \cite[Proposition 2.6]{DO12} implies the double branched cover gives a rational spin homology $D^4$ bounded by $\Sigma(L)$ with a spin structure corresponding to the orientation of $L$.   
Then \cref{theo: main ineq d bar} and \cite[Theorem 1.1 (iv)]{KMT} implies 
\[
0 \leq \delta_R (L), 0 \leq \underline{\delta}_R(L), 0 \leq \bar{\delta}_R(L), 0\leq \kappa(L).  
\]
Also, if $L$ is $\chi$-slice then, $-L^*$ is also $\chi$-slice. It implies that $\underline{\delta }_R(L)$, $\bar{\delta}_R (L)$ and $\kappa(L)$ are zero. 

Suppose $L$ and $L'$ are $\chi$-concordant. Then $-L^*_0 \# L_1$ is $\chi$-slice. This implies 
$\delta_R(-L^*_0 \# L_1)=0$. On the other hand, since $\delta$ is homomorphism, we see $\delta_R (L_0)=\delta_R(  L_1)$. 
\end{proof}
As a corollary of \cref{conn sum of delta}, we have;  
\begin{cor}
The invariant $\delta_R (L)$ gives a homomorphism from $\tilde{\mathcal{F}}$ to $\frac{1}{16}\Z$.
\end{cor}

\subsection{$\kappa $ invariants for links} 
We gave a generalization of the kappa invariant for knots defined in \cite{KMT} to that of links with non-zero determinant in \cref{def for link}: 
\[
L \mapsto \kappa (L) \in \frac{1}{16 }\Z .
\]
We also note general properties of the invariant for links. 
In particular, we provide 10/8-inequality for surfaces bounded by links as (iv) in the the following theorem. 
\begin{theo}\label{K general ineq for link}
The invariant $\kappa$ satisfies the following properties: 
\begin{itemize}
    \item[(i)] $2\kappa (L) \equiv -\frac{\sigma(L)}{8} \mod 2$,
    \item[(ii)] 
$
\kappa (L) + \kappa (-L^* ) \geq 0$, 
\item[(iii)] Let $L$ and $L'$ be oriented links in $S^3$, $X$ be an oriented smooth compact connected cobordism from $S^3$ to $S^3$ with $H_1(X; \Z)=0$, and $S$ be an oriented compact connected properly and smoothly embedded cobordism in $X$ from $L$ to $L'$ such that the homology class $[S]$ of $S$ is divisible by $2$ and $PD(w_2(X)) = [S]/2 \operatorname{mod} 2$. Then, we have
    \begin{align} \label{ineq1}
  -\frac{\sigma(X)}{8} + \frac{9}{32}S \circ S-\frac{9}{16}\sigma(L') + \frac{9}{16}\sigma(L) \leq b^+(X) + g(S) + \kappa(L')-\kappa(L). 
    \end{align}
\end{itemize}

\end{theo}

\begin{proof}
 For a proof of (i), we take a properly and smoothly embedded and connected oriented surface $S$ in $D^4$ bounded by $L$. For such a surface, the signature of $L$ is given by 
$\sigma( \Sigma(S) )$ from \cref{branched homology}. 
Thus, one can describe the Rochlin invariant of $\Sigma(L)$ by $\frac{1}{8}\sigma( \Sigma(S) )$. 
Then Theorem 1.1 (i) in \cite{KMT} implies 
\[
-2 \kappa (L) = \mu ( \Sigma(L), \mathfrak{t}_L) = \frac{1}{8}\sigma( \Sigma(S) ). 
\]
This completes the proof of (i). 
 The inequality (ii) followed from \cite[Theorem 1.1 (iii)]{KMT}. The inequality (iii) follows by combining \cref{branched homology} with \cite[Theorem 1.1(iv)]{KMT}.  
\end{proof}

\section{Applications} \label{section: application} 
In this section, we prove applications in the introduction such as non-smoothable and non-orientable surfaces in 4-manifolds \cref{theo: unorientable surfaces}, obstruction to the Nielsen realization problem \cref{theo:  Nielsen} and non-orientable surfaces in $D^4$ bounded by torus knots \cref{torus non orientable}. 

\subsection{Applications to non-smoothable actions}

We start with proving an application to non-smoothable actions,  \cref{theo: nonsmoothable action}.

\begin{proof}[Proof of \cref{theo: nonsmoothable action}]
First, we shall construct a locally linear involution $\iota : W \to W$.
Define a topological 4-manifold $W'$ by
\[
W' = mS^{2} \times S^{2} \# k(-\CP^{2}) \# 2(-E_{8}) \# 2n S^{2} \times S^{2}.
\]

By Freedman theory \cite{Fre82}, we have a homeomorphism $h :W \to W'$.
Define orientation-preserving diffeomorphisms $f_{0} : S^{2} \times S^{2} \to S^{2} \times S^{2}$ and $f_{1} : -\CP^{2} \to -\CP^{2}$ by $f_{0}(x,y)=(y,x)$ and $f_{1}([z_{0}:z_{1}:z_{2}])=[\bar{z}_{0} : \bar{z}_{1} : \bar{z}_{2}]$.
Both of $f_{0}$ and $f_{1}$ have codimention-2 fixed-point set, diffeomorphic to $S^{2}$ and $\RP^{2}$ respectively, and we may form an equivariant connected sum
\[
f_{\#} = \#_{m}f_{0}\#_{k}f_{1} : mS^{2} \times S^{2} \# k(-\CP^{2}) \to mS^{2} \times S^{2} \# k(-\CP^{2}).
\]
Choose a point $x_{0} \in mS^{2} \times S^{2} \# k(-\CP^{2})$ outside the fixed-point set of $f_{\#}$, and attach two copies of $(-E_{8})\#nS^{2} \times S^{2}$ along $x_{0}$ and $f_{\#}(x_{0})$.
Now we obtain a locally linear involution $f'$ of $W'$, and define $\iota = h^{-1} \circ f' \circ h$.

It is straightforward to check that 
\[
b^{+}_{f_{0}}(S^{2} \times S^{2}) = 1,\quad
b^{-}_{f_{0}}(S^{2} \times S^{2}) = 0,\quad
b^{+}_{f_{1}}(-\CP^{2}) = 0,\quad
b^{-}_{f_{1}}(-\CP^{2}) = 0,\quad
\]
and thus we have
\begin{align}
\label{eq: cal Betti 1}
b^{+}_{\iota}(W) = m,\quad
b^{-}_{\iota}(W) = n+8.
\end{align}

Let $c' \in H^{2}(W';\Z)$ be a characteristic element defined by
\[
c' = (0_{S}, E_{1}, \ldots, E_{k}, 0_{ES}),
\]
where $0_{S}$ and $0_{ES}$ denote the zero elements of $H^{2}(mS^{2} \times S^{2})$ and $H^{2}(2(-E_{8})\#2n S^{2} \times S^{2})$ respectively, and $E_{i}$ are copies of a generator of $H^{2}(-\CP^{2})$.
Let $c = h^{\ast}c'$ and let $\fraks$ be a spin$^{c}$ structure on $W$ with $c_{1}(\fraks)=c$.
Since $f_{1}^{\ast}E=-E$, we have that $\iota^{\ast}\fraks = \bar{\fraks}$.

Since $\iota$ has codimenion-2 fixed-point set, by \cref{lem :codimension 2} $\iota$ is of odd type.
Hence it follows from \eqref{eq: cal Betti 1} and  and \cref{theo: main} that $c_{1}(\fraks)^{2}-\sigma(W) \leq 0$.
However, since we have $c_{1}(\fraks)^{2} - \sigma(W)  = 16$ by a direct calculation, this is a contradiction.

\end{proof}

\subsection{Applications to non-smoothable and non-orientable surfaces}
\label{subsec Applications to non-smoothable and non-orientable surfaces}

The above argument on non-smoothable actions can be translated to a result on smoothable and unorientable surfaces,  \cref{theo: unorientable surfaces}:

\begin{proof}[Proof of \cref{theo: unorientable surfaces}]
Let us use the notation of the statement and the proof of \cref{theo: nonsmoothable action}.
Then it follows that $X$ is homeomorphic to $W/\iota$.
Let $S$ be the fixed-point set of $\iota : W \to W$.
Then $S$ is homeomorphic to $k\RP^{2}$ and is locally-flatly embedded in $W$, and hence also in $X$.
If $S$ is topologically isotopic to a smoothly embedded surface, it induces a smooth odd involution that acts on homology just as $\iota$ does, but it is a contradiction by 
the proof of \cref{theo: nonsmoothable action}.

We shall see what the homology class of $S$ is.
As the fixed-point set of $f_0 : S^2 \times S^2 \to S^2 \times S^2$ is the diagonal, which descends to a 2-sphere in $\CP^2$ that represent $2[\CP^1] \in H_2(\CP^2)$.
On the other hand, $-\CP^2/f_1$ is diffeomorphic to $S^4$.
Thus we have that the homology class of $S$ is zero in $H_2(X;\Z/2)$.

The normal Euler number $[k\RP^2]^2$ can be immediately deduced from \cref{branched homology}.
\end{proof}

\subsection{Applications to the Nielsen realization problem}
\label{section Applications to the Nielsen realization problem}

In this \lcnamecref{subsec Applications to non-smoothable and non-orientable surfaces}, we shall prove a result on the Nielsen realization problem for non-spin 4-manifolds, \cref{theo: Nielsen}.
First, we recall the definition of a diffeomorphism called the {\it reflection} about a $(-1)$-sphere in a 4-manifold:

\begin{defi}
Let $W$ be a smooth oriented 4-manifold $W$ and $S$ a $(-1)$-sphere $S$ in $W$.
Then one may form an orientation-preserving diffeomorphism \[
\rho_{S} : W \to W
\]
called the {\it reflection}, constructed as follows:
first, let $\rho : \CP^2 \to \CP^2$ be the complex conjugation, $[z_0:z_1:z_2] \mapsto [\bar{z}_0 : \bar{z}_1 : \bar{z}_2]$.
Take an embedded 4-disk $D^4$ in $\CP^2$.
Let $\rho'$ be a diffeomorphism that fixes $D^4$ pointwise obtained by deforming $\rho$ near $D^4$ by isotopy.
By reversing orientation, it induces a diffeomorphism $\rho' : -\CP^2 \to -\CP^2$.
In general, for a $(-1)$-sphere $S$ in $W$, a tubular neighborhood of $S$ is diffeomorphic to a tubular neighborhood of $-\CP^1$ in $-\CP^2$, which is the fixed point set of $\rho$.
Since a tubular neighborhood of $-\CP^1$ in $-\CP^2$ is diffeomorphic also to $\CP^2\setminus \mathrm{int}D^4$, we may implant $\rho'$ into $W$ by extending it by the identity, which we denote by $\rho_S : W \to W$.
\end{defi}

To prove our results on the Nielsen realization problem, we need to take care of Dehn twitsts.
First, let us recall the definition of the Dehn twist along an embedded annulus:

\begin{defi}
Let $W$ be a 4-manifold and let $A \cong S^3 \times [0,1]$ be an embedded annulus in $W$.
Then one may define an orientation-preserving diffeomorphism 
\[
\tau : W \to W
\]
called the {\it Dehn twist along $A$} as follows.
Define the {\it model Dehn twist} $\tau_0 : S^3 \times [0,1] \to S^3 \times [0,1]$ by $\tau_0(y,t)=(l(t)\cdot y,t)$, where $l : [0,1] \to SO(4)$ is the homotopically non-trivial loop in $SO(4)$.
As $\tau_0$ is supported inside $S^3 \times (0,1)$, one may extend $\tau_0$ by the identity and get a diffeomorphism $\tau : W \to W$.
\end{defi}

Note that, by $\pi_1(SO(4))=\Z/2$, the square of the Dehn twist $\tau$ along an annulus is smoothly isotopic to the identity through an isotopy supported inside the annulus.

Given a closed smooth 4-manifold $W$, let $\Diff(W,D^4)$ denote the group of diffeomorphisms of $W$ that are identity on a fixed embedded 4-disk $D^4 \subset W$.
Let 
\[
p : \pi_0(\Diff(W,D^4)) \to \pi_0(\Diff(W))
\]
be the map induced from the natural inclusion $\Diff(W,D^4) \hookrightarrow \Diff(W)$.
Note that $p$ is surjective since any diffeomorphism can be isotoped to a diffeomorphism that fixes $D^4$ pointwise.

\begin{lem}[{Giansiracusa~\cite[Corollary~2.5, Proposition~3.1]{Gi08}}]
\label{lem: Gian}
Let $W$ be a simply-connected closed smooth 4-manifold $W$.
Then we have the following:
\begin{enumerate}
\item $\ker{p}$ is isomorphic to either $\Z/2$ or $0$, and it is generated by the Dehn twist around the boudnary $\del(W \setminus D^4)$.
\item If $W$ is diffeomorphic to $W'\#\CP^2$ or $W'\#(-\CP^2)$ for some 4-manifold $W'$, then $\ker{p}=0$.
\end{enumerate} 
\end{lem}

In summary, one has an exact sequence 
\[
0 \to \ker{p} \to \pi_0(\Diff(W,D^4)) \xrightarrow{p} \pi_0(\Diff(W)) \to 0
\]
such that $\ker{p}$ is either $\Z/2$ or $0$ generated by the Dehn twist, and $\ker{p}=0$ if $W$ contains either $\CP^2$ or $-\CP^2$ as a connected sum factor.

Using \cref{lem: Gian}, we see that the reflection about a $(-1)$-sphere gives an order 2 mapping class:

\begin{lem}
\label{lem: reflection order 2}
For an oriented 4-manifold $W$ and a $(-1)$-sphere $S$ in $W$,
the mapping class $[\rho_S]$ of the reflection is of order 2 in $\pi_0(\Diff(W))$.
\end{lem}

\begin{proof}
First, clearly $\rho$ is an order 2 element in $\Diff(\CP^2)$, and so is the mapping class $[\rho]$ in $\pi_0(\Diff(\CP^2))$.
By \cref{lem: Gian}, the natural surjection
\[
p : \pi_0(\Diff(\CP^2, D^4))
\to \pi_0(\Diff(\CP^2)))
\]
is isomorphic.
This implies that the relative mapping class $[\rho']$ is of order 2 in $\pi_0(\Diff(\CP^2, D^4))$.
The claim of the \lcnamecref{lem: reflection order 2} now follows immediately by the construction of $\rho_S$.
\end{proof}

We shall also use the following \lcnamecref{lem: K simultaneous twist} about simultaneous Dehn twists on punctured $S^4$:

\begin{lem}[{\cite[Lemma~4.3]{K22}}]
\label{lem: K simultaneous twist}
Let $N >0$.
Define $S^{4}_{0}$ to be a $N$-punctured 4-sphere, $S^{4}_{0} = S^{4} \setminus (\sqcup_{i=1}^{N}D_{i}^{4})$.
Let $\tau_{S^{4}_{0}} : S^{4}_{0} \to S^{4}_{0}$ be a diffeomorphism defined by performing the Dehn twists near all $\del D_{i}^{4}$ simultaneously.
Then $\tau_{S^{4}_{0}}$ is smoothly isotopic to the identity in the group $\Diff(\del S^{4}_{0}, \del S^{4}_{0})$ of diffeomorphisms which fix the boundary pointwise.
\end{lem}

Now we are ready to prove the main result in this \lcnamecref{subsec Applications to non-smoothable and non-orientable surfaces}, \cref{theo: Nielsen}:

\begin{proof}[Proof of \cref{theo: Nielsen}]
Suppose that $\rho_{S}$ is homotopic to a smooth involution $\iota : W \to W$.
Then we have $\iota^{\ast} \fraks \cong \bar{\fraks}$.

%We claim that $\iota^{\ast} \fraks \cong \bar{\fraks}$. To see this, note that $\rho_S$ is the identity apart from a tubular neighborhood $\nu_S$ of $S$. On the other hand, $\nu_S$ is diffeomorphic to $-\CP^2 \setminus \mathrm{Int}(D^4)$, and $\rho_S$ acts on $H^2(-\CP^2 \setminus \mathrm{Int}(D^4))$ as multiplication by $-1$.
%Decomposing 

%Regard $\nu_S$ as an oriented plane bundle over $S$. Thus it suffices to prove that, given a spin$^c$ structure $\frakt$ on the sphere bundle $S(\nu_S)$, an extension of $\frakt$ to the disk bundle $D(\nu_S)$ is unique. By the relative obstruction theory, such extensions is classified by $H^2(D(\nu_S), S(\nu_S);\Z)$, which is zero by the Thom isomorphism for $\nu_S \to S$.

%Then we have that $\iota^{\ast} \fraks \cong \bar{\fraks}$, since a spin$^{c}$ structure is determined by $c_{1}$ by the assumption that $H^{2}(W;\Z)$ has no 2-torsion.

Recall that the induced action of $\rho_{S}$ on homology $H_{2}(W)$ is given by $(\rho_{S})_{\ast}(x) = x+2(x\cdot[S])[S]$.
Set $\varphi = (\rho_{S})_{\ast} : H_{2}(W) \to H_{2}(W)$.
Then we have
\[
b^{+}_{\varphi}(W) = b^{+}(W),\quad
b^{-}_{\varphi}(W) = b^{-}(W)-1,\quad
\sigma_{\varphi}(W) = \sigma(W)+1.
\]
Here $b^{+}_{\varphi}(W)$ and $b^{-}_{\varphi}(W)$ denote the maximal dimension of the $\varphi$-invariant part of $H^+(W)$ and of $H^-(W)$ respectively, and we set $\sigma_\varphi(W) = b^{+}_{\varphi}(W) - b^{-}_{\varphi}(W)$.
By the assumption that $\sigma(W) \neq -2$, we have $\sigma_{\varphi}(W) \neq \sigma(W)/2$.
It follows from \cref{lem: odd sign half} that $\iota$ is of odd type.
Combined this with $b^{+}_{\varphi}(W) = b^{+}(W)$, 
we may apply \cref{theo: main} to $\iota$, but
this contradicts the assumption that $c_{1}(\fraks)^{2} -\sigma(W)>0$.
\end{proof}

\begin{proof}[Proof of \cref{theo: smooth vs. top Nielsen}]
As the case that $r=0$ has been treated in \cite[Theorem~1.3]{K22}, we suppose that $r>0$.
Set $W' = pK3\#qS^2\times S^2$.
In \cite[Proof of Theorem~1.3]{K22}, a locally linear topological involution $f : W' \to W'$ with the following properties was constructed:
\begin{itemize}
\item $b^+_f(W') = 3p+q$, $\sigma_f(W')=-5p+q$,
\item the fixed point set of $f$ is non-empty and of codimension-2.
\item there is a diffeomorphism $g : W' \to W'$ such that $g^2$ is smoothly isotopic to the identity and $g_\ast=f_\ast$ on $H_2(W';\Z)$.
\end{itemize}

Let $\rho : -\CP^2 \to -\CP^2$ be the complex conjugation.
The fixed point set of $\rho$ is of codimension-2: it is given by $\RP^2 \subset -\CP^2$.
Choose points in $\RP^2$, and
form an equivariant connected sum of $r$-copies of $\rho$ along those points to get an involution on $r(-\CP^2)$, denoted by $\#_r \rho$.

On the other hand, recall that the fixed point set of $f$ is also codimension-2.
Also, since $f$ is locally linear, we have the notion of tangential representation of $f$ at a fixed point of $f$.
Since $f$ is of order 2 and has codimension-2 fixed point set, the representation is $\mathrm{diag}(-1,-1,1,1)$ in some coordinate, which coincides with the tangential representation of $\#_r \rho$ at a fixed point.
Thus we may form an equivariant connected sum of $f$ with $\#_r \rho$, which we denote by $\tilde{f} : W \to W$.

Let $g' : W' \to W'$ be a diffeomorphism obtained by deforming $g$ by isotopy so that $g'$ has a pointwise fixed 4-disk $D^4_1 \subset W'$.
%It follows from \cref{lem: Gian} that the mapping class $[(g')^2]$ regarded as an element of $\pi_0(\Diff(W',D^4))$ is the Dehn twist.
Similarly, by deforming $\#_r \rho$ by isotopy, we may get a diffeomorphism $g'' : r(-\CP^2) \to r(-\CP^2)$ that has a fixed 4-disk $D^4_2$.
Let $S^4_0$ be the 2-punctured 4-sphere, $S^4_0 = S^4 \setminus (\sqcup_{i=1}^2 D_i^4)$.
Regard $W$ as the connected sum of $W'$ with $r(-\CP^2)$ along $D^4_1$ and $D^4_2$, 
\[
W = (W' \setminus D^4_1) \cup_{\del D^4_1} S^4_0 \cup_{\del D^4_2} (r(-\CP^2)\setminus D^4_2).
\]
Form a diffeomorphism $\tilde{g} : W \to W$ by gluing $g'$ with $g''$ along $D^4_1$ and $D^4_2$.
Evidently, $\tilde{g}$ is supported outside $S^4_0$.
By construction, the induced action of $\tilde{g}$ on homology coincides with that of $\tilde{f}$.
This combined with a result by Quinn~\cite{Qu86} and Perron~\cite{P86} implies that $\tilde{g}$ is topologically isotopic to $\tilde{f}$.
Also, it is easy to see that
\begin{align}
\label{eq: action on homology of tilde g}
b^+_{\tilde{g}}(W) = 3p+q,\quad \sigma_{\tilde{g}}(W) = -5p+q
\end{align}
using the action of $f$ on homology.

We claim that $\tilde{g}^2$ is smoothly isotopic to the identity.
First, let $\tau_1$ and $\tau_2$ denote the Dehn twists on $W$ along $\del D^4_1$ and $\del D^4_2$ respectively.
Recall that $g^2$ is smoothly isotopic to the identity and $(\#_r \rho)^2=1$.
It follow from this together with (1) of \cref{lem: Gian} that $\tilde{g}^2$ is smoothly isotopic to one of the following: $\tau_1 \circ \tau_2$, $\tau_1$, $\tau_2$ and $\id$.
Also, by (2) of \cref{lem: Gian}, $\tau_2$ is smoothly isotopic to the identity in $\Diff(r(-\CP^2) \setminus \mathrm{Int}D^4_2, \del D^4_2)$.
So it suffices to prove that $\tau_1$ is smoothly isotopic to the identity.
However, \cref{lem: K simultaneous twist} implies that $\tau_1$ is smoothly isotopic to $\tau_2$, which is smoothly isotopic to the identity again by (2) of \cref{lem: Gian}.
This proves the claim.

Let $G$ be the subgroup of $\pi_0(\Diff(W))$ generated by the mapping class of $\tilde{g}$.
By the above claim, $G$ is of order 2.
As $\tilde{g}$ is topologically isotopic to $\tilde{f}$, the group $G'$ defined as the image of $G$ under the map $\pi_0(\Diff(W)) \to \pi_0(\Homeo(W))$ is realized in $\Homeo(W)$.
Also, the group $G'$ is non-trivial since the action of $\tilde{f}$ on homology is non-trivial.
This proves the statement on $G'$ in the \lcnamecref{theo: smooth vs. top Nielsen}. 

What remains to show is that $G$ cannot be realized in $\Diff(W)$.
This is equivalent to prove that $\tilde{g}$ is not smoothly isotopic to a smooth involution.
Suppose that $\tilde{g}$ is smoothly isotopic to a smooth involution $\iota : W \to W$.
Let $\fraks$ be the spin$^c$ structure on $W$ obtained as the connected sum of the unique spin structure on $W'$ with $r$-copies of the spin$^c$ structure on $-\CP^2$ whose first Chern class generates $H^2(-\CP^2)$.
Then it is obvious that $\iota^\ast \fraks \cong \bar{\fraks}$ and $c_1(\fraks)-\sigma(W)>0$.
It follows from this combined with \cref{lem: odd sign half} and \eqref{eq: action on homology of tilde g} that $\iota$ is of odd type.
Again by \eqref{eq: action on homology of tilde g}, we have $b^+(W)=b^+_\iota(W)$.
Thus we have a contradiction by \cref{theo: main}.
This proves the claim and completes the proof of the \lcnamecref{theo: smooth vs. top Nielsen}.
\end{proof}

\subsection{Several inequalities for knots}
\label{subsec Several inequalities for knots}

In this \lcnamecref{subsec Several inequalities for knots},
we give a proof of \cref{unoriented bound} which describes a bound for the first Betti numbers of non-orientable surfaces embedded into $D^4$. 
\begin{proof}[Proof of \cref{unoriented bound}]
Let $S$ be a (possibly non-orientable) smoothly and properly embedded connected surface in $D^4$ bounded by a given knot $K$. 
Suppose the determinant of $K$ is $1$, the Manolescu--Owen invariant $\delta(\Sigma (K))$ \cite{MO07} is zero and $\delta_R(K)<0$.
As a known topological obstruction (which can be regarded as $b^+(S) \geq 0$), we have 
\[
\left|\sigma(K)- \frac{1}{2}S \circ S\right|  \leq b_1(S). 
\]
Thus it is sufficient to suppose 
\[
-\sigma(K)+ \frac{1}{2}S \circ S= b_1(S)
\]
and obtain a contradiction. In order to use \cref{theo: main knot}(iii), we verify the assumptions. We first note that \eqref{assumption delta}
follows from $  -\sigma(K)+ \frac{1}{2}S \circ S= b_1(S)$. 
Since $K$ is determinant-1, $\Sigma(K)$ is a homology 3-sphere. So, the intersection form of $\Sigma(S)$ is unimodular. From \cref{branched homology}, we know 
\[
-\sigma(\Sigma(S))= \frac{1}{2}S \circ S -\sigma(K)= b_1(S) = b_2(\Sigma(S)) .
\]
This implies the intersection form of $\Sigma(S)$ is negative definite. Now, since we supposed the Manolescu--Owen invariant  $\delta(\Sigma (K))$ is zero, from Fr\o yshov inequality and Elikies's theorem, we know that the intersection form of $\Sigma(S)$ is the direct sum of copies of $ (-1)$.
From \cref{branched homology}, we know that the double branched covering space $\Sigma(S)$ has no 2-torsions in $H^2(\Sigma(S); \Z)$. 
Thus, the set of isomorphism classes of spin$^c$ structures on $\Sigma(S)$ is determined by their first Chern classes in $H^2(\Sigma(S); \Z)$. Note that the action of $\iota^*$ on $H^2(\Sigma(S); \Z)$ gives a decomposition \[
H^2(\Sigma(S); \Z)= H^2(\Sigma(S); \Z)^{\iota^*} \oplus H^2(\Sigma(S); \Z)^{-\iota^*}
\]
as the eigenvalue decomposition of $\iota$ corresponding to $+1$ and $-1$ respectively. Note that the dimension $\dim H^2(\Sigma(S); \Z)^{\iota^*} =b_2(\Sigma(S)/\iota)= b_2(D^4)=0$. This implies $H^2(\Sigma(S); \Z)= H^2(\Sigma(S); \Z)^{-\iota^*}$. So, by taking a spin$^c$ structure corresponding to the sum of basis of $\bigoplus (-1)$ and by applying \cref{theo: main knot}(iii), we have 
\[
0 \leq \delta(K). 
\]
This gives a contradiction to $\delta(K)<0$.
\end{proof}

We next prove crossing change formulae for $\delta_R$, $\bar{\delta}_R$ and $\underline{\delta}_R$. 
\begin{proof}[Proof of \cref{crossing change}]
Let $K$ and $K'$ be knots in $S^3$. Suppose $K'$ is obtained from $K$ by a positive crossing change. By blown up, we obtain a smooth annulus cobordism $S$ in $I \times S^3 \# -{\CP}^2$ from $(S^3, K)$ to $(S^3, K')$ such that 
\[
[S]=2 \in \Z = H_2(-{\CP}^2)= H_2(I \times S^3 \#  -{\CP}^2).
\]
Suppose $2=\sigma(K)-\sigma(K)$. Then, for the unique spin structure on $\Sigma(S)$, the assumptions of \cref{theo: main knot}(iii) are satisfied. Then \cref{theo: main knot}(iii) implies the desired inequality. 
In the case of $\sigma (K)- \sigma (K')=0$, we use $\overline{I \times S^3 \# (-{\CP}^2)}$ instead of $I \times S^3 \# (-{\CP}^2)$ and \cref{theo: main knot}(iii) again. 
\end{proof}

\subsection{Applications to non-orientable genera and normal Euler numbers} 

\begin{proof}[Proof of \cref{torus non orientable}]
Note that $K=T(3, 6n+1)$ is determinant $1$, the Manolescu--Owen's invariant is known to be zero and $\delta_R( T(3, 6n+1))=- \frac{1}{2}<0$. Thus all assumptions of \cref{unoriented bound} are satisfied. 
For a smooth embedded connected surfaces in $D^4$ bounded by $K$ satisfying 
\[
(e, h) =(8/3 (1-n) +2+2m  ,1+m),
\]
we can check $-\sigma(K)+ \frac{1}{2}S \circ S = b_1(S)$. 
However, from \cref{unoriented bound}, we have the stronger inequality $-\sigma(K)+ \frac{1}{2}S \circ S + 1 \leq b_1(S)$ which gives a contradiction. This completes the proof.    
\end{proof}

\bibliographystyle{plain}
\bibliography{tex}

\end{document}